\documentclass[onefignum,onetabnum]{siamart190516}
\usepackage{amsfonts}
\usepackage{graphicx}
\usepackage{bm}
\usepackage{mathtools}
\usepackage{amsmath}
\usepackage{amssymb}
\usepackage{hyperref,bookmark}
\usepackage{enumitem}
\usepackage[algo2e]{algorithm2e}
\usepackage{enumitem}
\DeclarePairedDelimiter{\norm}{\lVert}{\rVert}
\def\[#1\]{\begin{align*}#1\end{align*}}

\DeclareRobustCommand{\vect}[1]{\bm{#1}}
\newsiamremark{remark}{Remark}
\newsiamremark{hypothesis}{Hypothesis}
\crefname{hypothesis}{Hypothesis}{Hypotheses}
\newsiamthm{claim}{Claim}

\graphicspath{{figs/multiplier/}}

\headers{Conservative Integrators for PWS with Transversal Dynamics}
{A. Hirani, A. Wan and N. Wojtalewicz}

\title{Conservative Integrators for Piecewise Smooth Systems with Transversal Dynamics}

\author{Anil N. Hirani\thanks{Department of Mathematics, University of Illinois at Urbana-Champaign, Urbana, IL
    (\email{hirani@illinois.edu, npw2@illinois.edu}).}
  \and Andy T. S. Wan\thanks{Department of Mathematics, University of
    Northern British Columbia, Prince George, BC
  (\email{andy.wan@unbc.ca}).}
  \and Nikolas Wojtalewicz\footnotemark[1]
}

\usepackage{amsopn}

\ifpdf
\hypersetup{
  pdftitle={Conservative Integrators for Piecewise Smooth Systems with Transversal Dynamics},
  pdfauthor={A. N. Hirani, A. T.S. Wan and N. Wojtalewicz}
}
\fi

\begin{document}

\maketitle

\begin{abstract}
We introduce conservative integrators for long term integration of piecewise smooth systems with transversal dynamics and piecewise smooth conserved quantities. In essence, for a piecewise dynamical system with piecewise defined conserved quantities such that its trajectories cross transversally to its interface, we combine Mannshardt's transition scheme and the Discrete Multiplier Method to obtain conservative integrators capable of preserving conserved quantities up to machine precision and accuracy order. We prove that the order of accuracy of the conservative integrators is preserved after crossing the interface in the case of codimension one number of conserved quantities. Numerical examples illustrate the preservation of accuracy order and conserved quantities across the interface.
\end{abstract}



\section{Introduction}\label{sec:intro}
Piecewise defined ordinary differential equations (ODEs) arise naturally in a variety of fields including electrical circuits, mechanical systems, biology and others. Electrical circuits with an ideal diode have different dynamics based on whether the diode is conducting electricity or not~\cite{AcBoBr2011}.  In mechanical systems with friction, the Coulomb model of friction allows for a sudden change in the frictional force depending on whether the object is moving or stationary. In gene regulatory networks the dynamics depends on which genes are turned on or off~\cite{PlKj2005}. 

We develop a conservative numerical method for integrating ODEs corresponding to piecewise defined vector fields. Our results depend only on the local piecewise structure of the vector fields. Locally, we assume the phase space is subdivided into two subdomains separated from each other by a smooth codimension 1 embedded submanifold called the \emph{switching surface}, or equivalently the \emph{interface}.  For example, in two dimensions, the phase space consists of two subdomains is separated by a simple smooth curve and in three dimensions, the phase space consists of two subdomains is separated by smooth surface. We do not address the case of higher codimension switching submanifolds such as a curve singularity in three dimensions~\cite{GuHa2017}.

A piecewise defined vector field is defined on both subdomains with a different definition on each side of the interface. The vector field is not required to be defined on the switching surface. However, in our subsequent analysis, we require that the vector fields from the two sides be continuous up to the interface, without requiring their limit from both sides agree on the interface. Moreover, we will require the vector field to be \emph{transverse} to the surface and point in such a way as to allow the piecewise trajectories to continue after arriving at the surface. Thus, we do not consider \emph{sliding motion} on the switching surface~\cite{PlKj2005}. Furthermore, we assume the smooth vector fields defined on each subdomain satisfy the standard conditions needed for local existence and uniqueness of solutions.

Our main goal is long term integration of such vector fields across repeated crossings of the switching surface, while preserving one or more conserved quantities, even though the conserved values may be different in each of the two subdomains. Indeed, even the conservation laws can be different in each subdomain.

The theory of ODEs with discontinuous right hand sides is well-developed. Fillipov and others~\cite{Filippov1988} have generalized the definition of solution of such an ODE by replacing the equation with a differential inclusion (so that at the discontinuities the vector field belongs to a set rather than being equal to a specific value) and investigated conditions for existence and uniqueness and dependence on parameters and initial conditions. The appearance of chaos and bifurcation theory for such systems has been another direction of investigation~\cite{BeBuChKo2008}. Numerical methods for such ODEs have also been developed~\cite{Mannshardt1978,GeOs1984,AcBr2008}. For a survey of such methods see~\cite{DiLo2012}. 
An alternative approach is in~\cite{FeMaOrWe2003} where nonsmooth Lagrangian mechanics is discretized  using a discrete variational principle leading to geometric variational integrators. The emphasis in~\cite{FeMaOrWe2003} is on rigid body motion with collisions. Their approach works when a Lagrangian formulation is available and discrete conservation of the energy, momentum or symplectic form is sought. In contrast, our integrator does not need a Lagrangian formulation and is applicable for arbitrary conserved quantities. Another difference between our work and~\cite{FeMaOrWe2003} is our emphasis on analysis of the order of accuracy before and after the interface is encountered.

The main contribution of this paper is the development of a \emph{conservative} numerical method with accuracy guarantees, for piecewise smooth dynamics in which there are possibly different conservation laws before and after a trajectory encounters the switching surface. We do this by extending the \emph{Discrete Multiplier Method} (DMM)~\cite{WaBiNa2017}  for dynamical systems with conserved quantities to yield conservative numerical schemes which allow for repeated interface crossings and jumps in the conserved quantity across the switching surface. For long term integration with conserved quantities of codimension 1 and a time-independent interface, our method guarantees the correct trajectory of the dynamical system with error only in time, even after repeated crossings of the interface. We show that the accuracy order of our method is preserved on crossing the interface by applying \emph{Mannshardt's transition scheme}~\cite{Mannshardt1978} to conservative methods. Furthermore, we show that  geometrical properties inherent to conservative methods can lead to a simpler proof of order preservation.

\subsection{Background on piecewise smooth systems}
\label{subsec:background}

First, we introduce some notations and then discuss in detail the piecewise smooth (PWS) systems we are interested throughout this paper.
Let $d, d'$ be positive integers and let $U\subset \mathbb{R}^{d}$ and $V\subset \mathbb{R}^{d'}$ be open subsets. $C(U\rightarrow V)$ denotes the space of continuous functions and $Lip(U\rightarrow V)$ denotes the spaces of Lipschitz continuous functions from $U$ to $V$. For a positive integer $k$, $C^k(U\rightarrow V)$ denotes the space of functions all of whose partial derivatives up to $k$-th order exist and are continuous from $U$ to $V$. A closed ball of radius $r>0$ with a center at $\bm{a}\in \mathbb{R}^d$ will be denoted $B_r(\bm{a})$. 

The setting we are interested in is a PWS system on a bounded open interval $I$ with its phase space on a bounded open subset $U\subset \mathbb{R}^d$ separated into two parts by a codimension 1 hypersurface. Specifically, a function $g\in C^1(U\rightarrow \mathbb{R})$ with $\nabla g \neq \bm{0}$ on $U$ is called a switching function and defines a switching surface $S := \{\bm{x} \in U \;\vert\; g(\bm{x}) = 0\}$, which is a $C^1$ $(d-1)$-dimensional embedded submanifold of $U \subset \mathbb{R}^d$. We denote the two parts of the phase space as $U_{\pm}$, where $U_+ := \{ \bm{x} \in U \; \vert \; {g}(\bm{x}) > 0\}$ and $U_- := \{ \bm{x} \in U \; \vert \; {g}(\bm{x}) < 0\}$, and assume $U_{\pm} \neq \varnothing$\footnote{If either of $U_{\pm}$ is empty, then the dynamical system is smooth.}. Moreover, we assume there are two vector fields $\bm{f}_\pm \in C^1(I\times (U_\pm\cup S)\rightarrow \mathbb{R}^d)$ which defines a PWS system of the form:
\begin{align}
  \label{eq:typeA} \dot{\bm{x}}(t) &=
  \left\{\begin{aligned}
    \bm{f}_-(t, \bm{x}(t)),\quad & \text{if} \quad \bm{x}(t)\in U_-,\\
    \bm{f}_+(t, \bm{x}(t)),\quad & \text{if} \quad \bm{x}(t)\in U_+\, ,
  \end{aligned}\right.\\
  \bm{x}(t_0) &= \bm{x}_0 \in U\setminus S\, .  \nonumber
\end{align}
As $\dot{\bm{x}}(t)$ is undefined on $\bm{x}(t)\in S$, we need to decide on the dynamics upon reaching the interface $S$. We shall assume that the trajectories are \emph{transversal} to the switching surface and are such that the trajectories approaching the switching surface from one subdomain will exit into the other subdomain. This condition, known as the \emph{transversality condition} is satisfied if there exists a nonzero $\alpha_S\in \mathbb{R}$, referred to as the \emph{transversality constant}\footnote{In the notation for the transversality constant, we emphasize the dependence on $S$ but not on time. In the analysis the key estimates  since we will be interested in crossing $S$ in a particular time interval.}, so that,
\begin{equation}
\big(\nabla {g} \cdot \bm{f}_\pm\big)(t^*,\bm{x}^*)\geq  \alpha_S^2\, , \text{ for any } (t^*,\bm{x}^*)\in I\times S\, .
  \label{eq:transverseA}
\end{equation}
For example, this corresponds to a phase plot such as the first two plots in Figure~\ref{fig:vftypes}.
(We have defined transversality as both ($\pm$) terms $\big(\nabla {g} \cdot \bm{f}_\pm\big)$ being positive. To accommodate both being negative one can redefine the switching function $\tilde{g}:=-g$.)
\begin{figure}[ht]
  \centering
  \includegraphics[width=1in,trim=1in 0in 1in 0in, clip]
  {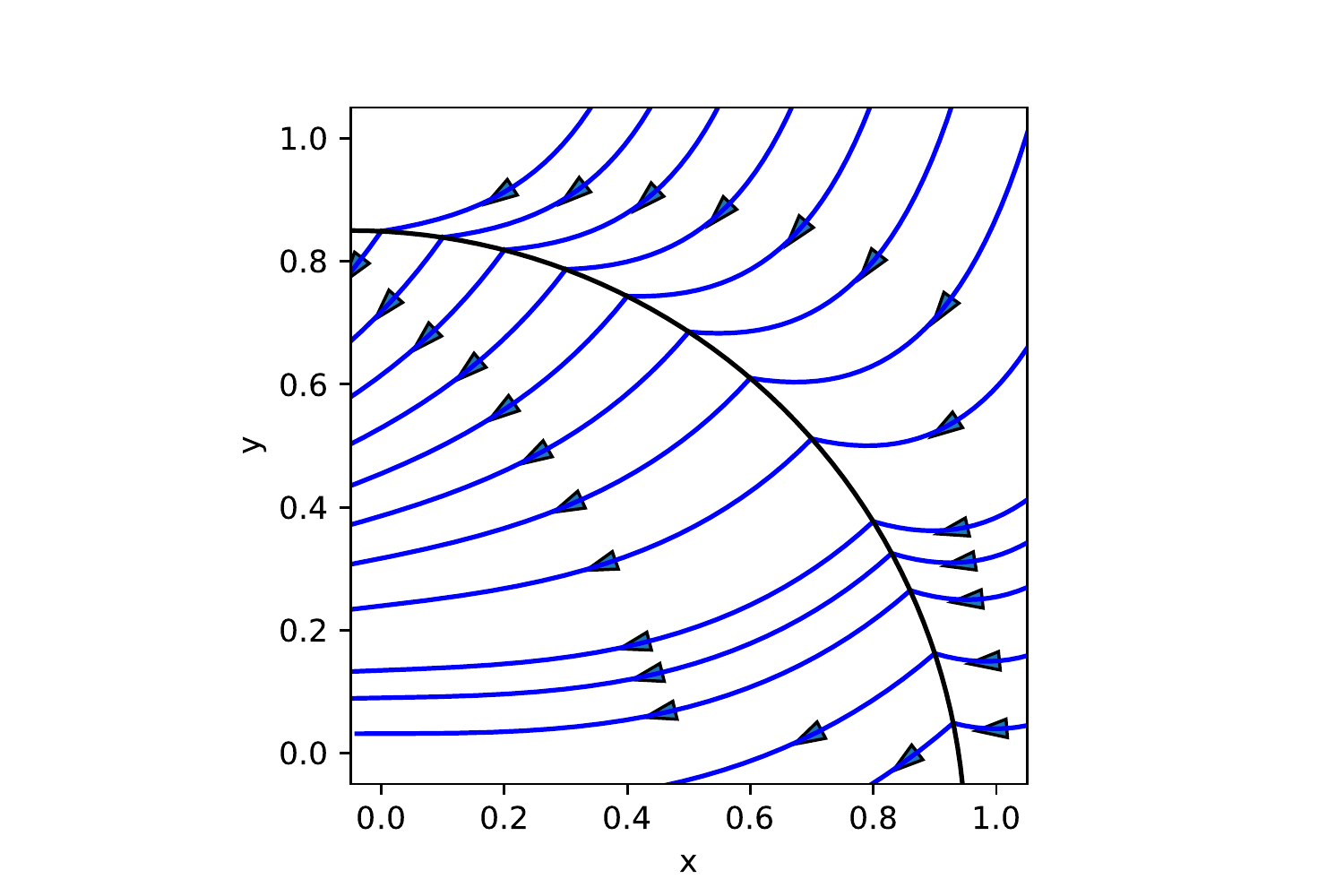}
  \includegraphics[width=1in,trim=1in 0in 1in 0in, clip]
  {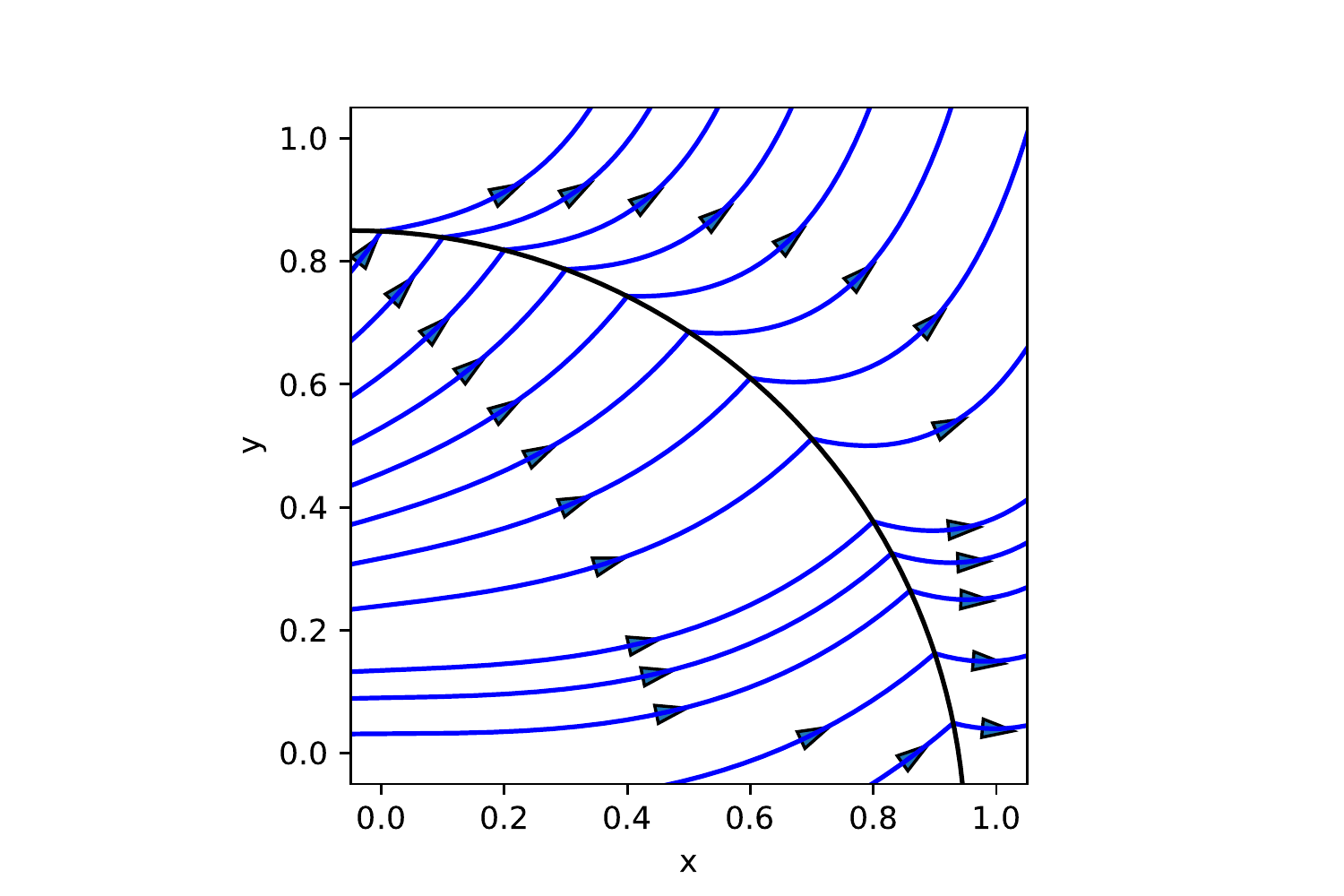}
  \includegraphics[width=1in,trim=1in 0in 1in 0in, clip]
  {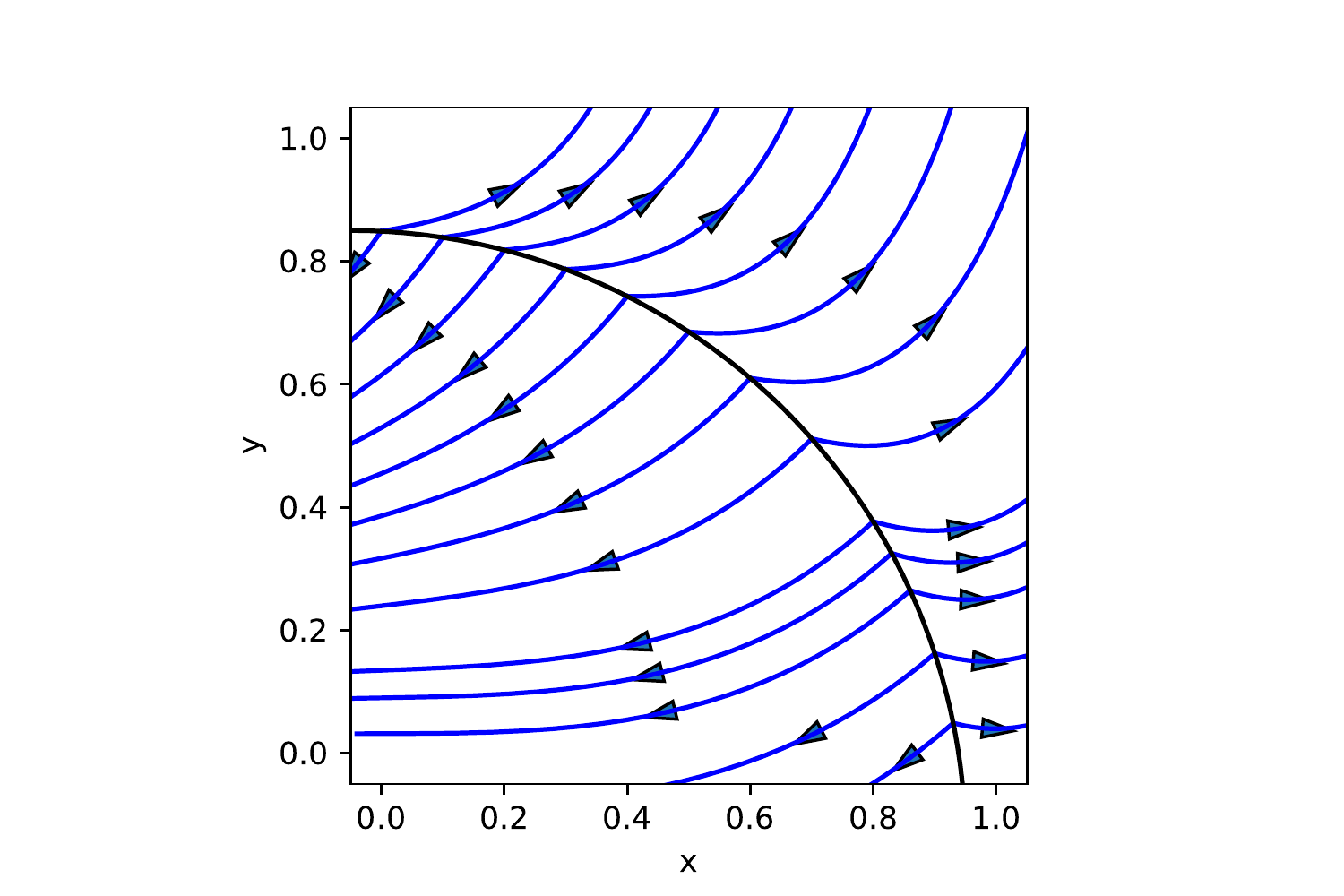}
  \includegraphics[width=1in,trim=1in 0in 1in 0in, clip]
  {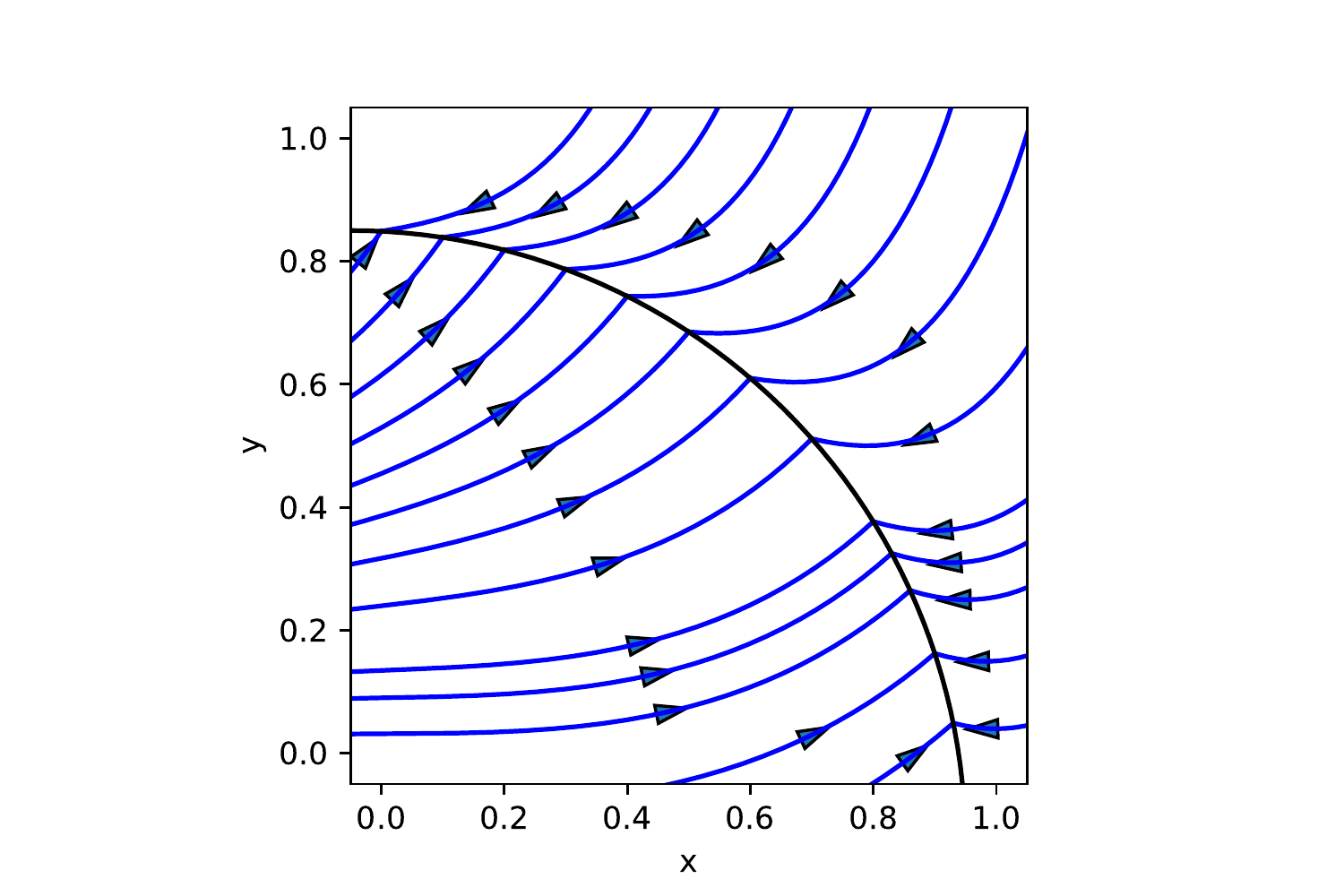}
  \caption{Pictorial classification of piecewise smooth vector fields inspired by Figure~6.2 of~\cite{HaNoWa2008}. PWS dynamics described by~\cref{eq:typeA} with transversality condition~\cref{eq:transverseA} corresponds to dynamics of type shown in the first two figures. These are the piecewise vector fields addressed in this paper. The third figure corresponds to an unstable dynamics on the switching surface and the fourth figure requires a choice in determining sliding dynamics along the switching surface~\cite{Filippov1988}. We do not consider these types of vector fields.}\label{fig:vftypes}
\end{figure}
From Filippov's first order theory \cite{Filippov1988} on PWS systems, if the transversality condition is satisfied, any trajectory $\bm{x}(t)$ of \eqref{eq:typeA} not starting on the interface and reaching $S$ at $t^*$ will exit $S$ after $t^*$ and is unique. Moreover, we wish to exclude the possibility of infinitely many transitions across $S$ over a finite time interval. For this reason and for later proofs, throughout this paper, we assume that the solutions $\bm{x}\in C(I\rightarrow U)$ of \eqref{eq:typeA} satisfy the following hypotheses: 
\begin{enumerate}[label=(H\arabic*)]
\item There exist at most finitely many \emph{transition times} $t^*\in I$ so that $\bm{x}(t^*)\in S$. \label{hyp:H1}
\item \label{hyp:H2} The one-sided limits of $\dot{\bm{x}}$ approaching $S$ exist and are equal to $\bm{f}_\pm$ from either $U_\pm$. Specifically,  there exists a fixed $\epsilon_0>0$ such that for $0<\epsilon<\epsilon_0$,
	\[\lim_{t \downarrow t^{*}} \dot{\bm{x}}(t) &= \bm{f}_\pm (t^*, \bm{x}^*), \text{ if } \bm{x}(t^* + \epsilon) \in U_\pm,\\
	\lim_{t \uparrow t^{*}} \dot{\bm{x}}(t) &= \bm{f}_\pm (t^*, \bm{x}^*), \text{ if } \bm{x}(t^* - \epsilon) \in U_\pm.\] 
\item \label{hyp:H3} Denoting the finitely many transition times $t_i^\ast$ as $a\leq t_1^*<\dots<t_r^*\leq b$ with $[a,b]\subset I$, $\bm{x}\in C^2(I\setminus \{t_1^*,\ldots,t_r^*\}\rightarrow U)$ and $\norm{\ddot{\bm{x}}(t)}_2<\infty$ for $t \in I \setminus \{t_1^*,\ldots,t_r^*\}$.
\end{enumerate}There are other possibilities for introducing non-smoothness in dynamical systems. For example, the vector field can change to a different one based on reaching some particular value of a function of the state variable. Or the state can be instantaneously transported along the switching surface and the trajectory can restart from a new point on the surface. A bouncing ball is an example of this type. We do not address the analysis of these cases in this paper, although numerically, the strategy we adopt for solving this class of problems is the same as the one we do address in our analysis.

\subsection{Numerical strategies for piecewise smooth systems}

There are two main challenges in developing numerical methods for piecewise smooth systems. One challenge is determining accurately when and where the trajectory intersects the interface. A second challenge is to preserve the order of accuracy of the method when the numerical solution crosses the interface. Indeed, as Mannshardt states in his seminal paper~\cite{Mannshardt1978} on the topic: ``... Almost every Runge-Kutta method remains convergent after transition, but only with order 1 ...''.

There are several possibilities for the design of numerical methods for piecewise smooth systems. (An overview of these methods is covered in the review paper \cite{DiLo2012}.) One simple strategy is to test at each step whether the interface has been crossed but ignore the issue of order preservation. A second strategy is to use a variable step-size code with local error control (such as DOPRI5, MATLAB ode45 etc.). Such methods rely on the error estimation built into the integrator to adapt the time step and cross the surface without needing to compute its precise location, as long as there is a test available to decide which side of the surface the integrator is on.  There may be more rejected steps near $S$ and a negative impact on accuracy and efficiency as can be seen in~\cite{HaNoWa2008} in Figure~6.3 and related discussion.

One approach for such systems is to reparameterize time so that the unknown transition time is transformed into a known time. See~\cite{DiLo2015,LoMa2018} for details of such methods. Yet another strategy is the one that Mannshardt~\cite{Mannshardt1978} used and analyzed and this is the one we adopt in this paper. The idea is to arrive at the switching surface $S$ using one right hand side, stop the integrator just prior to arriving at $S$ and then restart using the other right hand side. Our method falls into this class of methods, which are also called \emph{event driven} methods for piecewise smooth systems. The original main results in using this strategy are in~\cite{Mannshardt1978} which showed that the accuracy of an integrator after crossing $S$ is limited by how accurately the arrival point on $S$ is computed. See~\cite{DiLo2012} for a description of the essence of Mannshardt's method. We will refer to Mannshardt's method as a \emph{transition scheme}. Our goal is to extend this approach for order preservation, while preserving conserved quantities.


\subsection{Conserved Quantities and Conservative Integrators}
A conserved quantity is a function that is constant along the integral curves of the ODE and an ODE may possess several conserved quantities. A well-known example is the total energy $\psi(x, v) = \frac{1}{2}(m v^2 + k x^2)$ which is constant along the solutions of the harmonic oscillator $v = \dot{x}, \dot{v} = -\frac{k}{m} x$. In ODEs arising from a variational principle, a common method for obtaining conserved quantities in the presence of symmetries, is by the use of Noether's theorem~\cite{Olver1993}. See also~\cite{BlAn2002} for other techniques for finding conserved quantities. Conserved quantities (also referred to as constants of motion) may also exist for dissipative systems. For example, $\psi(t, x, v) = \frac{1}{2}(m v^2 + \gamma x v + k x^2)\exp({\gamma t/m})$ is constant along the solutions of the damped harmonic oscillator $v = \dot{x}, \dot{v} = -\frac{\gamma}{m}v - \frac{k}{m} x$, see~\cite{WaBiNa2017}. We will be using conservative integrators based on DMM~\cite{WaBiNa2016, WaBiNa2017, WaNa2018}. These integrators preserve conserved quantities up to machine precision.  Moreover, DMM has recently been applied to many body problems~\cite{gormezano2021conservative,wan2021conservative}.


We only consider the PWS system \eqref{eq:typeA} with solutions satisfying \ref{hyp:H1}-\ref{hyp:H3} that have time-independent conserved quantities $\bm{\psi}_\pm\in C^1(U_\pm \cup S \rightarrow  \mathbb{R}^{d_\pm})$, which are linearly independent\footnote{By linear independence, we mean that $\nabla \bm{\psi}_\pm$ have full row rank on $U_\pm$, see Definition 11 of~\cite{WaBiNa2017}}. In general, our numerical approach is well-defined for any number $d_\pm\geq 1$ of conserved quantities. In the particular case when there are $d_\pm=d-1$ conserved quantities, we prove that the accuracy order is preserved upon crossing the interface.

\subsection{Summary}
In this paper, we introduce a novel conservative numerical method for piecewise smooth systems. We prove the preservation of accuracy order and demonstrate preservation of conserved quantities for this method. This paper is organized as follows. In Section \ref{sec:NotationConventionAlgorithm}, we introduce the notations and conventions used throughout the paper. Furthermore, we also introduce the algorithm used to transition past the interface. In Section \ref{subsec:order_preservation}, we present the main result of the paper, Theorem \ref{thm:mainResults}. Specifically, we present several lemmas for estimating the error of transition time leading to the proof of the main result. In Section \ref{sec:numerical}, we present numerical results demonstrating preservation of conserved quantities over long term integration and preservation of accuracy order of our method over many transitions.


\section{Conservative Integrator for PWS systems}\label{sec:NotationConventionAlgorithm}
In this section, after introducing some notations, we will introduce the conservative integrator and discuss various related well-posedness results.
\subsection{Notations}
\label{subsec:notation}
Throughout the rest of the paper, $\tau$ will refer to the maximum over all discrete time step sizes and $\norm{\cdot}$ denotes the $\ell_2$ norm of vectors. For simplicity and convenience, we will assume that the time step size is fixed although all results apply with minimal changes to the variable time step case. Let $T$ denote the final time and define uniformly distributed time steps $t_k = t_0 + k\tau$, $N = (T - t_0) / \tau$, so that $T = t_0 + N\tau$ and $\tau = (T - t_0) / N$. In the expressions where $\tau$ appears as a superscript, it is intended to denote discrete expressions that appear in the numerical method and not to the specific length of the discrete time step. Thus, for example, $\bm{f}_\pm^\tau(t_k, \bm{x}_k, t_{k+1}, \bm{x}_{k+1})$ denotes ``discrete vector fields'' approximating the PWS vector fields $\bm{f}_\pm(t, \bm{x})$. The discrete solution $\bm{x}^\tau(t;\bm{x}_k, t_k)$ denotes the numerical solution at time $t$, starting with the initial condition $\bm{x}_k$ at $t_k$ and this is supposed to approximate the piecewise smooth solution $\bm{x}(t; \bm{x}_k, t_k)$ starting with the same initial conditions. Often when it is clear from the context, we will shorten these to $\bm{x}^\tau(t)$ and $\bm{x}(t)$, respectively. By hypothesis \ref{hyp:H3}, we can assume any solution curve will cross the interface at most finitely many times. We denote an exact transition time and location at which the interface is crossed by starred quantities with $t^*\in I$ and $\bm{x}^*\in U$ and the corresponding discrete transition time $\hat{t}^*\in I$ and point $\hat{\bm{x}}^*\in U$. Specifically,
\begin{align}
    \bm{x}^* &= \bm{x}(t^*; \bm{x}_k, t_k), \label{def:PWIntersect} \\
    \hat{\bm{x}}^* &= \bm{x}^\tau(\hat{t}^*; \bm{x}_k, t_k). \label{def:DiscIntersect}
\end{align}
\subsection{Conservative Transition Scheme}
\label{sec:conservative}
A general conservative method for smooth dynamical systems is one that can preserve any conserved quantities. These are distinct from other geometric methods like symplectic, or more generally variational integrators~\cite{MaWe2001, HaLuWa2006} which are designed to preserve properties like the symplectic form, energy and momentum related to the underlying physical systems. Some examples of general conservative methods are the projection method~\cite{HaLuWa2006}, the discrete gradient method~\cite{McQuRo1998, McQuRo1999} and the DMM~\cite{WaBiNa2017}. It is the DMM that we wish to extend for PWS systems in this paper.

We will consider single-step conservative numerical schemes for the PWS system \eqref{eq:typeA} of the form
\begin{equation}
  \bm{x}_{k+1} = \bm{x}_k + \tau \bm{f}^\tau_\pm(t_k, \bm{x}_k, t_{k+1}, \bm{x}_{k+1}) \; ,\label{eq:twoSchemes}
\end{equation}
such that $\bm{\psi}_\pm(\bm{x}_{k+1}) = \bm{\psi}_\pm(\bm{x}_k)$ where $\bm{\psi}_\pm(\bm{x})$ is constant along the integral curves of the corresponding smooth parts of \eqref{eq:typeA}. \\
We can consider any number $d_\pm$ of conserved quantities $\bm{\psi}_\pm\in C^1(U_\pm\rightarrow \mathbb{R}^{d_\pm})$. In the ideal case, when there are $d_\pm = d-1$ conserved quantities, the shape of the trajectory is completely determined by the conserved quantities and the conservative transition scheme will only induce an error in time. 
Using this geometric property, we will show in Section \ref{subsec:order_preservation} that the resulting conservative scheme can preserve the order of accuracy on crossing the switching surface with a simplified proof than the work of Mannshardt \cite{Mannshardt1978}. \\
In order to preserve the accuracy order, a special treatment for \eqref{eq:twoSchemes} is needed when the discrete solution crosses the interface $S$, which is detected by evaluating ${g}(\bm{x}_k)$ after each discrete time step. If such a crossing is detected at time step $t_{k+1}$, then the method backs up to time step $t_k$ and follows the transition scheme introduced next.\\
Assuming $(t_k,\bm{x}_k) \in I \times U_-$\footnote{Analogous steps can be made if $(t_k,\bm{x}_k)\in I\times U_+$ by interchanging $\bm{f}^\tau_\pm$.}, find $(\hat{t}^*,\hat{\bm{x}}^*)$ by solving for $(t,\bm{x})$ in
\begin{align}
    &\bm{x} = \bm{x}_k + (t - t_k)\;\bm{f}^\tau_-(t_k,
    \bm{x}_k, t, \bm{x})\, ,   \label{eq:to-interface}\\
    &g(\bm{x}) = 0 \label{eq:interface}\,.
\end{align}  The error incurred in the numerical solution $(\hat{t}^\ast, \hat{\bm{x}}^\ast)$ of the system~\eqref{eq:to-interface}-\eqref{eq:interface} plays an important role in the analysis of the order preservation across the interface, detailed in Section~\ref{subsec:order_preservation}.
  \begin{remark}
    Note that we only consider time-independent  $g$ and so the interface is fixed. The transition scheme can still work in principle for the time-dependent case $g(t,\bm{x})$. However, our main result in Section~\ref{subsec:order_preservation} on the preservation of order across the interface is only proved for the time-independent case. Preserving order across a time-dependent interface would require the ability to control the time error between the discrete and continuous trajectories. Unless one can control $\hat{t}^\ast - t^\ast$, the error in the time of intersection with the interface, it is difficult to control $\hat{\bm{x}}^\ast - \bm{x}$, the error in the location of the intersection.
  \end{remark}
Assuming a solution $(\hat{t}^\ast, \hat{\bm{x}}^\ast)$ is found from \eqref{eq:to-interface} and \eqref{eq:interface}, then the transition scheme finishes by solving for
$\bm{x}_{k+1}$ in
\begin{equation}
  \bm{x}_{k+1} = \hat{\bm{x}}^\ast + (t_{k+1} - \hat{t}^\ast)\; \bm{f}^\tau_+(\hat{t}^\ast,
  \hat{\bm{x}}^\ast, t_{k+1}, \bm{x}_{k+1})\, .
  \label{eq:next}
\end{equation}
Together,~\eqref{eq:to-interface}, ~\eqref{eq:interface} and~\eqref{eq:next} can be combined to yield the following transition step across the interface:
\[
  \bm{x}_{k+1} = \bm{x}_k + (\hat{t}^\ast - t_k) \bm{f}^\tau_-(t_k, \bm{x}_k, \hat{t}^\ast,\hat{\bm{x}}^\ast) + (t_{k+1} - \hat{t}^\ast) \bm{f}^\tau_+(\hat{t}^\ast,\hat{\bm{x}}^\ast, t_{k+1},\bm{x}_{k+1})\, .
\]
This can be interpreted as a special time stepping scheme that takes a fractional time step to the interface and then another fractional time step to complete the step from $t_k$ to $t_{k+1}$ which forms a convex combination of the two schemes given by
\[
  \frac{\bm{x}_{k+1}-\bm{x}_k}\tau =
  \frac{\hat{t}^\ast - t_k}\tau \bm{f}^\tau_-(t_k, \bm{x}_k, \hat{t}^\ast,\hat{\bm{x}}^\ast) +
  \frac{t_{k+1}-\hat{t}^\ast}\tau \;
  \bm{f}^\tau_+(\hat{t}^\ast,\hat{\bm{x}}^\ast, t_{k+1},\bm{x}_{k+1})\, .
\]
We summarize Mannshardt's transition scheme as follows:
\begin{algorithm}[h]
\SetAlgoLined
 Given $\tau,t_0,\bm{x}_0$.\\
 \For{$k=0,1,2,\dots$}{
 $t_{k+1} = t_k+\tau$\\
 $s\leftarrow \text{sign}(g(\bm{x}_{k}))$\\
 $\tilde{\bm{x}}_{k+1} \leftarrow$ Solve for $\bm{x}$ in $\boxed{\bm{x}=\bm{x}_k+\tau \bm{f}_{s}^\tau(t_k,\bm{x}_k,t_{k+1},\bm{x})}$\\
 $s'\leftarrow \text{sign}(g(\tilde{\bm{x}}_{k+1}))$\\
  \eIf{$s = s'$}{
    $\bm{x}_{k+1}\leftarrow \tilde{\bm{x}}_{k+1}$
  }{
    $(\hat{t}^*,\bm{\hat{x}}^*)\leftarrow$ Solve for $(t,\bm{x})$ in  $\boxed{\begin{cases}\bm{x}=\bm{x}_k+(t-t_k) \bm{f}_{s}^\tau(t_k,\bm{x}_k,t,\bm{x})\\ g(\bm{x})=0\end{cases}}$\\
    $\bm{x}_{k+1} \leftarrow$ Solve for $\bm{x}$ in $\boxed{\bm{x}=\bm{\hat{x}}^*+(t_{k+1}-\hat{t}^*) \bm{f}_{s'}^\tau(\hat{t}^*,\bm{\hat{x}}^*,t_{k+1},\bm{x})}$
   }
 }
 \caption{Mannshardt's transition scheme (uniform time steps)} \label{alg:trans_scheme}
\end{algorithm}

The main steps in Algorithm~\ref{alg:trans_scheme} are the solution of the system~\eqref{eq:to-interface} and~\eqref{eq:interface} for $(\hat{t}^\ast, \hat{\bm{x}}^\ast)$ followed by computing the update $\bm{x}_{k+1}$ in~\eqref{eq:next}. The existence and uniqueness of these is discussed in detail in Appendix~\ref{app:existence_intersection_points}, the contents of which we briefly summarize here.  Lemma~\ref{lem:suff_cond_transition}  states a sufficient condition for existence of $(t^\ast$, $\bm{x}^\ast)$ the exact transition time and point for the piecewise smooth solution $\bm{x}(t)$. A partial converse of this requires transversality and is the content of Lemma~\ref{lem:part_conv_transition}. Assuming that $\hat{t}^\ast$ exists during a single time step $[t_k, t_{k+1}]$ the remainder of Appendix~\ref{app:existence_intersection_points} is devoted to proving the existence of $(\hat{t}^\ast, \hat{\bm{x}}^\ast)$ and $\bm{x}_{k+1}$. The existence result is in Prop.~\ref{prop:existence_numerical_transition} which needs two technical Lemmas~\ref{lem:wellposed-lemma1} and~\ref{lem:wellposed-lemma2}. These lemmas prove the existence of a certain fixed point (which is equivalent to solving~\eqref{eq:to-interface}) and the Lipschitz continuity of the discrete solution. To the best knowledge of the authors, these results have not appeared previously in the literature, even though the discrete equations \eqref{eq:to-interface}-\eqref{eq:next} go back to the original work of Mannshardt \cite{Mannshardt1978}.

\section{Order preservation upon crossing the interface}
\label{subsec:order_preservation}

\pagebreak
In this section, we will present our main result on the order preservation for Mannshardt's transition scheme with conservative integrators. We first state our main result concerning preservation of order. Following series of lemmas in Sections \ref{sec:quadLemma} and \ref{sec:mainEst}, we then prove the main theorem in Section \ref{sec:mainResults}.

\begin{theorem}[Main Theorem]\label{thm:mainResults}
 Let the piecewise smooth trajectory $\bm{x}(t;\bm{y},s)$ be the solution of~\eqref{eq:typeA},  $\bm{x}^\tau(t;\bm{y},s)$ be the solution of the transition scheme~\eqref{eq:to-interface}-\eqref{eq:next} and
$g\in C^2(\overline{U}\rightarrow \mathbb{R})$ denote the switching function. Assume that the discrete vector fields $\bm{f}_\pm^\tau \in C(I\times (U_\pm\cup S)\times I\times (U_\pm\cup S)\rightarrow \mathbb{R}^d)$ are Lipschitz continuous in their arguments. Moreover, assume the piecewise smooth trajectory $\bm{x}(t;\bm{y},s)$ and the discrete trajectory $\bm{x}^\tau(t;\bm{y},s)$ are Lipschitz continuous with respect to their initial data $\bm{y},s$.
Furthermore, assume that the transition schemes \eqref{eq:to-interface}-\eqref{eq:next} are conservative for $d-1$ time-independent conserved quantities $\bm{\psi}_\pm$ of the piecewise smooth system~\eqref{eq:typeA}, where $\nabla \bm{\psi}_\pm$ have full row rank in $U_\pm \cup S$, respectively. If $\hat{t}^*\in [t_k,t_{k+1}]$, then for sufficiently small $\tau$,
\begin{equation*}
    \norm{\bm{x}(t_{k+1}; \bm{x}_k, t_k) - \bm{x}^\tau(t_{k+1}; \bm{x}_k, t_k)} = \mathcal{O}(\tau^p).
\end{equation*}
\end{theorem}
We note that our main result is a statement about the local error upon crossing the interface on some time interval $[t_k,t_{k+1}]$. Specifically, assuming we have order $p$ schemes in the smooth regions, then the global error is also of order $p$, provided that the local error is also of order $p$ upon crossing the interface a fixed number of times. Note that we do not need to show that the local error on crossing is order $p+1$ due to the assumption of finitely many crossings on $I$.
\subsection{An elementary lemma on quadratic inequalities}
\label{sec:quadLemma}

We begin with an elementary lemma about nonnegative solutions to a quadratic inequality that we will utilize for our main theorem.

\begin{lemma}\label{lem:quadraticIneq}
    Consider the inequality $0 \leq a x^2 - b x + c$ for $x \in [0,\infty)$, where $a,b$ are fixed positive constants and $c$ as a positive parameter. Then for $c < \frac{b^2}{4a}$, both roots of the quadratic polynomial are real and the nonnegative solutions to the inequality belong to the intervals,
    \begin{equation*}
        x \in \left[0, r_-(a,b,c) \right] \cup \left[ r_+(a,b,c), \infty \right), \text{ where } r_\pm(a,b,c):=\frac{b \pm \sqrt{b^2 - 4ac}}{2a}.
    \end{equation*}
    Furthermore, $\displaystyle r_-(a,b,c) \leq \frac{c}{b} + \frac{c}{b} \sum_{n=1}^\infty \left( \frac{2ac}{b^2} \right)^n = \frac{c}{b} + \mathcal{O}(c^2)$.
\end{lemma}

\begin{proof}
    Since $c < \frac{b^2}{4a}$, there are two real roots $r_\pm$ to the above quadratic polynomial. As $a>0$, the quadratic polynomial is concave up and so the nonnegative solutions to the quadratic inequality belong to $[0, r_-]\cup [r_+, \infty)$. Moreover, note that,
    \begin{align*}
        r_-(a,b,c)&=\frac{b - \sqrt{b^2 - 4ac}}{2a} = \frac{b^2 - (b^2 - 4ac)}{2a \left(b + \sqrt{b^2 - 4ac} \right)} = \frac{2c}{b}\frac{1}{1+\sqrt{1 - \frac{4ac}{b^2}}} \\
            &< \frac{2c}{b} \frac{1}{2 - \frac{4ac}{b^2}} = \frac{c}{b} \frac{1}{1 - \left( \frac{2ac}{b^2} \right)} = \frac{c}{b} + \frac{c}{b} \sum_{n=1}^\infty \left( \frac{2ac}{b^2} \right)^n = \frac{c}{b} + \mathcal{O}(c^2),
    \end{align*}
    where the inequality follows from $\sqrt{1-y} > 1 - y$ for $0 < y < 1$ and the geometric series converges since $c < \frac{b^2}{4a} < \frac{b^2}{2a}$.
\end{proof}

\subsection{Main estimates}
\label{sec:mainEst}

Next we prove two lemmas for estimating $\vert t - t^\ast \vert$ and $\vert t - \hat{t}^\ast\vert$ (which we will call times from crossing) during a single time step in which an interface crossing occurs. Further, we also prove that a \emph{discrete} transversality condition holds, which we utilize to prove the lemma estimating $\vert t - \hat{t}^* \vert$. The discrete transversality condition and both estimates are given in terms of various constants of the problem.

\begin{lemma}\label{lem:ctsTimeError}
    Let $g\in C^2(\overline{U}\rightarrow \mathbb{R})$ and assume the continuous transversality condition \eqref{eq:transverseA} is satisfied.
    Then for all $t\in[t_k,t_{k+1}]$, the following estimate holds:
    \begin{equation} \label{eq:crossing_estimate}
        |t - t^\ast| \leq \frac{M(t-t^\ast)^2 + L_g\norm{\bm{x}(t;\bm{x}_k,t_k) - \bm{x}(t^\ast;\bm{x}_k,t_k)}}{\alpha_S^2},
    \end{equation}
    where $L_g$ is the uniform Lipschitz constant of $g$ and 
    \begin{multline}
    M := \frac{1}{2}\max \left\{\sup_{t\in [t_k,t^\ast)} \left|\dot{\bm{x}}(t)\cdot  H_g(\bm{x}(t))\dot{\bm{x}}(t)+\nabla g(\bm{x}(t))\cdot \ddot{\bm{x}}(t)\right|\right.,\\
    \left.\sup_{t\in (t^\ast,t_{k+1}]} \left|\dot{\bm{x}}(t)\cdot  H_g(\bm{x}(t))\dot{\bm{x}}(t)+\nabla g(\bm{x}(t))\cdot \ddot{\bm{x}}(t)\right|\right\}, \label{eq:M}
    \end{multline} where $H_g$ is the Hessian matrix of $g$.
\end{lemma}
\begin{proof}
    Note that the $M$ defined in equation~\eqref{eq:M} exists by Hypothesis~\ref{hyp:H3}.
    Without loss of generality, we can assume $t^\ast\neq t\in [t_k,t_{k+1}]$. Then by applying Taylor's remainder theorem in $t$ for $g(\bm{x}(t))$, there is some $\xi$ strictly between $t,t^*$ satisfying
    \[
        g(\bm{x}(t)) & = g(\bm{x}^\ast) + \nabla{g(\bm{x}^\ast)} \cdot \dot{\bm{x}}(t^\ast)(t-t^\ast)\\ & + \frac{1}{2}\left[\dot{\bm{x}}(\xi)\cdot  H_g(\bm{x}(\xi))\dot{\bm{x}}(\xi)+\nabla g(\bm{x}(\xi))\cdot \ddot{\bm{x}}(\xi)\right](t-t^\ast)^2.
    \]
    Thus, depending on if $t<t^\ast$ or $t>t^\ast$, we can replace $\dot{\bm{x}}$ by $\bm{f}_\pm$ in the second term above, since $\dot{\bm{x}}$ satisfies \eqref{eq:typeA} and \ref{hyp:H2}. Thus, bounding the second derivative terms by $M$ yields
    \begin{equation}\label{lem:ctsTimeErrorTaylorTime}
        \left|g(\bm{x}(t)) - \left(g(\bm{x}^\ast) + \nabla{g(\bm{x}^\ast)}\cdot \bm{f}_\pm(t^\ast,\bm{x}^\ast)(t-t^\ast)\right)\right| \leq M|t-t^\ast|^2.
    \end{equation}
    As $g\in C^2$ on a compact set $\overline{U}$, it is uniformly Lipschitz satisfying,
    \begin{equation}\label{lem:ctsTimeErrorGLipschitz}
        |g(\bm{x}(t)) -g(\bm{x}^\ast)| \leq L_g \norm{\bm{x}(t) - \bm{x}(t^\ast)}.
    \end{equation}
    Combining the two inequalities and the transversality condition \eqref{eq:transverseA} gives the resulting bound:
    \begin{align*}
        |t-t^\ast| &= \frac{|\nabla{g(\bm{x}^\ast)}\cdot \bm{f}_\pm(t^\ast,\bm{x}^\ast) (t-t^\ast)|}{|\nabla{g(\bm{x}^\ast)}\cdot \bm{f}_\pm(t^\ast,\bm{x}^\ast)|} \\
            &\leq \frac{\overbrace{|g(\bm{x}(t)) - (g(\bm{x}^\ast) +  \nabla{g(\bm{x}^\ast)}\cdot \bm{f}_\pm(t^\ast,\bm{x}^\ast) (t-t^\ast))|}^{\leq M|t-t^\ast|^2 \text{ by \eqref{lem:ctsTimeErrorTaylorTime}}} + \overbrace{|g(\bm{x}(t)) - g(\bm{x}^\ast)|}^{\leq L_g\norm{\bm{x}(t) - \bm{x}(t^\ast)}\text{ by \eqref{lem:ctsTimeErrorGLipschitz}}}}{\alpha_S^2}
    \end{align*}
\end{proof}

Before presenting a discrete version of Lemma \ref{lem:ctsTimeError}, we need to first show a result on the \emph{discrete transversality condition} for the discrete vector fields $\bm{f}^\tau_\pm$ in relation to $\nabla g$. The idea is based on the consistency of the discrete vector fields and the continuous transversality condition.

\begin{lemma}
\label{lem:discTC}
Let $g\in C^2(\overline{U}\rightarrow \mathbb{R})$. Suppose the discrete vector fields $\bm{f}^\tau_\pm$ are locally consistent of order $p$ to the respective vector fields $\bm{f}_\pm$ on $I\times (U_\pm\cup S)$; that is for $t\in I, \bm{x} \in U_\pm\cup S$, there exists positive constants $\tau_0, C_1, C_2, C_3$ so that for all $0<\tau\leq \tau_0$ and if $a,b\in B_{C_1\tau_0}(t)$ and $\bm{y}, \bm{z}\in B_{C_2\tau_0^p}(\bm{x})$, then
\begin{equation}\label{eqn:consistency}
\norm{\bm{f}_\pm(t,\bm{x})-\bm{f}^\tau_\pm(a,\bm{y},b,\bm{z})}\leq C_3\tau^p.
\end{equation}
Then for some positive constant $\tau_1$, for $0<\tau\leq \tau_1$, if $t,s\in B_{C_1\tau_1}(t^\ast)$ and $\bm{x},\bm{y},\bm{z}\in B_{C_2\tau_1^p}(\bm{x}^\ast)$, there exist a nonzero constant $\hat{\alpha}_S \in \mathbb{R}$ so that the following discrete transversality condition holds:
    \begin{equation}\label{eq:discrete_transversality}
        0 < \hat{\alpha}_S^2 \leq \nabla g(\bm{x}) \cdot \bm{f}^\tau_\pm(t, \bm{y}, s, \bm{z})
    \end{equation}
\end{lemma}
\begin{proof}
As $g\in C^2$ on a compact set $\overline{U}$, $\nabla g$ is uniformly Lipschitz satisfying,
    \begin{equation}\label{eqn:nablaGLipschitz}
        \norm{\nabla g(\bm{x}) -\nabla g(\bm{x}^\ast)} \leq L_{\nabla g} \norm{\bm{x} - \bm{x}^\ast}.
    \end{equation}
By the transversality condition \eqref{eq:transverseA}, consistency \eqref{eqn:consistency} and Lipschitz continuity of $\nabla g$ \eqref{eqn:nablaGLipschitz},
\[
0< \alpha_S^2 &\leq \nabla{g(\bm{x}^\ast)}\cdot \bm{f}_\pm(t^\ast,\bm{x}^\ast)\\
&= \nabla{g(\bm{x}^\ast)}\cdot \left(\bm{f}_\pm(t^\ast,\bm{x}^\ast)-\bm{f}^\tau_\pm(t, \bm{y}, s, \bm{z})\right)+ \left(\nabla{g(\bm{x}^\ast)}-\nabla{g(\bm{x})}\right)\cdot\bm{f}^\tau_\pm(t, \bm{y}, s, \bm{z})\\
&\quad+\nabla{g(\bm{x})}\cdot\bm{f}^\tau_\pm(t, \bm{y}, s, \bm{z})\\
&\leq \norm{\nabla{g(\bm{x}^\ast)}}\norm{\bm{f}_\pm(t^\ast,\bm{x}^\ast)-\bm{f}^\tau_\pm(t, \bm{y}, s, \bm{z})}+\norm{\nabla{g(\bm{x}^\ast)}-\nabla{g(\bm{x})}}\norm{\bm{f}^\tau_\pm(t, \bm{y}, s, \bm{z})}\\
&\quad +\nabla{g(\bm{x})}\cdot\bm{f}^\tau_\pm(t, \bm{y}, s, \bm{z})\\
&\leq C_3\norm{\nabla{g(\bm{x}^\ast)}}\tau^p+C_2 L_{\nabla g}M_1\tau^p+\nabla{g(\bm{x})}\cdot\bm{f}^\tau_\pm(t, \bm{y}, s, \bm{z}),
\]where $\displaystyle M_1:=\max_{t,s\in I, \bm{y},\bm{z}\in \overline{U}} \norm{\bm{f}^\tau_\pm(t, \bm{y}, s, \bm{z})}$.
Thus for sufficiently small $\tau_1>0$, for all $0<\tau\leq \tau_1$,
\[
0<\hat{\alpha}_S^2:=\alpha_S^2-\left(C_3\norm{\nabla{g(\bm{x}^\ast)}}+C_2 L_{\nabla g}M_1\right)\tau_1^p \leq \nabla{g(\bm{x})}\cdot\bm{f}^\tau_\pm(t, \bm{y}, s, \bm{z})
\]
\end{proof}

By analogy with the continuous case, we will refer to the constant $\hat{\alpha}_S$ appearing in the above lemma as the \emph{discrete transversality constant}.

\begin{lemma}\label{lem:discTimeError}
    Let $g \in C^2(\overline{U} \rightarrow \mathbb{R})$. For all $t \in [t_k, t_{k+1}]$, the following discrete analogue of the continuous estimate \eqref{eq:crossing_estimate} holds for $\tau \leq \tau_0$:
    \begin{equation} \label{eq:discTimeError}
        |t - \hat{t}^\ast| \leq \frac{\hat{M}(t - \hat{t}^\ast)^2 + L_g \norm{\bm{x}^\tau(t; \bm{x}_k, t_k) - \bm{x}^\tau(\hat{t}^\ast; \bm{x}_k, t_k)} }{\hat{\alpha}_S^2}
    \end{equation}
    Here $\hat{\alpha}_S$ is the discrete transversality constant and $\tau_0$ is an upper bound, both defined in Lemma~\ref{lem:discTC}. $L_g$ is the Lipschitz constant of $g$, and 
    \begin{equation}\label{eq:hatM}
      \hat{M} := \frac{1}2 \max (\hat{M}_-, \hat{M}_+)\, ,
    \end{equation}
    where
    \begin{align*}
      \hat{M}_- &:= \sup_{\substack{t \in [t_k, \hat{t}^\ast] \\ s \in [0,1]}}
      \bm{f}^\tau_-(t, \bm{x}^\tau(t), \hat{t}^\ast, \hat{\bm{x}}^\ast) \cdot
      H_g(\bm{x}^\tau(t) + s (\hat{\bm{x}}^\ast-\bm{x}^\tau(t) ))
      \bm{f}^\tau_-(t, \bm{x}^\tau(t), \hat{t}^\ast, \hat{\bm{x}}^\ast)\, ,\\
      \hat{M}_+ &:= \sup_{\substack{t \in [\hat{t}^\ast, t_{k+1}] \\ s \in [0,1]}}
      \bm{f}^\tau_+(\hat{t}^\ast, \hat{\bm{x}}^\ast, t, \bm{x}^\tau(t)) \cdot
      H_g(\hat{\bm{x}}^\ast + s (\bm{x}^\tau(t) - \hat{\bm{x}}^\ast))
      \bm{f}^\tau_+(\hat{t}^\ast, \hat{\bm{x}}^\ast, t, \bm{x}^\tau(t))\, .
    \end{align*}
\end{lemma}
\begin{proof}
    The proof for this inequality is analogous to the proof for the inequality of Lemma \ref{lem:ctsTimeError}. 
    First, since $g$ is $C^2$, for any $\bm{x},\bm{y}\in U$, by Taylor's theorem, there is some $s \in (0,1)$ such that:
    \begin{equation}\label{eq:gSpatialTaylorPoly}
        g(\bm{y}) = g(\bm{x}) + \nabla{g(\bm{x})} \cdot (\bm{y} - \bm{x}) + \frac{1}{2} (\bm{y} - \bm{x}) \cdot H_g(\bm{x} + s(\bm{y} - \bm{x})) (\bm{y} - \bm{x}).
    \end{equation}
    If $t<\hat{t}^\ast$, then by the definition of the transition scheme,
    \begin{equation*}
        \hat{\bm{x}}^\ast - \bm{x}^\tau(t) = (\hat{t}^\ast - t) \bm{f}^\tau_- (t,\bm{x}^\tau(t),\hat{t}^\ast,\hat{\bm{x}}^\ast)
    \end{equation*}
    So letting $\bm{x}:=\hat{\bm{x}}^\ast$ and $\bm{y}:=\bm{x}^\tau(t)$ in \eqref{eq:gSpatialTaylorPoly} gives
    \begin{equation}
    \begin{split}
    &g(\bm{x}^\tau(t))=g(\hat{\bm{x}}^\ast)+\nabla{g(\hat{\bm{x}}^\ast)} \cdot \bm{f}^\tau_-(t, \bm{x}^\tau(t), \hat{t}^\ast, \hat{\bm{x}}^\ast)(t-\hat{t}^\ast)\\&\quad+\underbrace{\frac{1}{2} \bm{f}^\tau_-(t, \bm{x}^\tau(t), \hat{t}^\ast, \hat{\bm{x}}^\ast) \cdot H_g(\bm{x}^\tau(t) + s (\hat{\bm{x}}^\ast-\bm{x}^\tau(t)))\bm{f}^\tau_-(t, \bm{x}^\tau(t), \hat{t}^\ast, \hat{\bm{x}}^\ast)}_{\leq \hat{M}}(t-\hat{t}^\ast)^2.
    \end{split}\label{eq:discG-}
    \end{equation}
    So by the discrete transversality condition \eqref{eq:discrete_transversality} 
    \[
    &\hat{\alpha}_S^2(\hat{t}^\ast-t)\leq \nabla{g(\hat{\bm{x}}^\ast)} \cdot \bm{f}^\tau_-(t, \bm{x}^\tau(t), \hat{t}^\ast, \hat{\bm{x}}^\ast)(\hat{t}^\ast-t)\\
    &\quad= g(\hat{\bm{x}}^\ast)-g(\bm{x}^\tau(t))\\
    &\quad\quad+\frac{1}{2} \bm{f}^\tau_-(t, \bm{x}^\tau(t), \hat{t}^\ast, \hat{\bm{x}}^\ast) \cdot H_g(\bm{x}^\tau(t) + s (\hat{\bm{x}}^\ast-\bm{x}^\tau(t)))\bm{f}^\tau_-(t, \bm{x}^\tau(t), \hat{t}^\ast, \hat{\bm{x}}^\ast)(t-\hat{t}^\ast)^2\\
    &\quad\leq L_g\norm{\hat{\bm{x}}^\ast-\bm{x}^\tau(t)}+\hat{M}(t-\hat{t}^\ast)^2,
    \] which implies the first part of the inequality upon dividing by $\hat{\alpha}_S^2$. Similarly, if $\hat{t}^\ast<t$, then again by the definition of the transition scheme,
    \begin{equation*}
        \bm{x}^\tau(t)-\hat{\bm{x}}^\ast = (t-\hat{t}^\ast) \bm{f}^\tau_+ (\hat{t}^\ast,\hat{\bm{x}}^\ast,t,\bm{x}^\tau(t))
    \end{equation*}
    So letting $\bm{x}:=\hat{\bm{x}}^\ast$ and $\bm{y}:=\bm{x}^\tau(t)$ in \eqref{eq:gSpatialTaylorPoly} gives
    \begin{equation}
    \begin{split}
    &g(\bm{x}^\tau(t))=g(\hat{\bm{x}}^\ast)+\nabla{g(\hat{\bm{x}}^\ast)} \cdot \bm{f}^\tau_+(t, \bm{x}^\tau(t), \hat{t}^\ast, \hat{\bm{x}}^\ast)(t-\hat{t}^\ast)\\&\quad+\underbrace{\frac{1}{2} \bm{f}^\tau_+(\hat{t}^\ast, \hat{\bm{x}}^\ast, t, \bm{x}^\tau(t)) \cdot H_g(\hat{\bm{x}}^\ast + s (\bm{x}^\tau(t) - \hat{\bm{x}}^\ast)) \bm{f}^\tau_+(\hat{t}^\ast, \hat{\bm{x}}^\ast, t, \bm{x}^\tau(t))}_{\leq \hat{M}}(t-\hat{t}^\ast)^2.
    \end{split}\label{eq:discGX-}
    \end{equation}
    So by the discrete transversality condition \eqref{eq:discrete_transversality}
    \[
    &\hat{\alpha}_S^2(t-\hat{t}^\ast)\leq \nabla{g(\hat{\bm{x}}^\ast)} \cdot \bm{f}^\tau_+(t, \bm{x}^\tau(t), \hat{t}^\ast, \hat{\bm{x}}^\ast)(t-\hat{t}^\ast)\\
    &\quad= g(\bm{x}^\tau(t))-g(\hat{\bm{x}}^\ast)\\
    &\quad\quad-\frac{1}{2} \bm{f}^\tau_+(\hat{t}^\ast, \hat{\bm{x}}^\ast, t, \bm{x}^\tau(t)) \cdot H_g(\hat{\bm{x}}^\ast + s (\bm{x}^\tau(t) - \hat{\bm{x}}^\ast)) \bm{f}^\tau_+(\hat{t}^\ast, \hat{\bm{x}}^\ast, t, \bm{x}^\tau(t))(t-\hat{t}^\ast)^2\\
    &\quad\leq L_g\norm{\bm{x}^\tau(t)-\hat{\bm{x}}^\ast}+\hat{M}(t-\hat{t}^\ast)^2,
    \] which implies the second part of inequality upon dividing by $\hat{\alpha}_S^2$. Combining the two cases gives the desired result.
\end{proof}


    

Lastly, we need one final lemma to prove the main theorem. Specifically, we prove a result relating the spatial transition points of the piecewise smooth and discrete trajectories.

\begin{lemma} \label{lem:geometricLemma}
    Let $\bm{f}^\tau_-$ be the conservative discretization of~\eqref{eq:to-interface} which preserves $d-1$ time-independent conserved quantities $\bm{\psi}_-:U_- \cup S \subset \mathbb{R}^d \rightarrow \mathbb{R}^{d-1}$. Recall from \eqref{def:PWIntersect}, $\bm{x}^* = \bm{x}(t^*; \bm{x}_k, t_k)$, and from \eqref{def:DiscIntersect}, $\hat{\bm{x}}^* = \bm{x}^\tau(\hat{t}^*; \bm{x}_k, t_k)$, which satisfies $g(\bm{x}^*) = 0 = g(\hat{\bm{x}}^*)$. Then for sufficiently small $\tau$, 
    \begin{equation} \label{eq:x^*Approx}
        \bm{x}^* = \hat{\bm{x}}^*.
    \end{equation}
\end{lemma}

\begin{proof}
    The main idea of the proof is to use the implicit function theorem to show that both trajectories $\bm{x}(\cdot)$ and ${\bm{x}^\tau}(\cdot)$ have a common intersection point, $\bm{x}^* = \hat{\bm{x}}^*$, on the interface $g=0$. 
    
    By the property of the conservative scheme and definition conserved quantities, $\bm{\psi}_- \left( \hat{\bm{x}}^* \right) = \bm{\psi}_- \left( \bm{x}_k \right) = \bm{\psi}_- \left( \bm{x}^* \right)$. Since $\nabla \bm{\psi}_-$ has full row rank in $U_- \cup S$, we can apply the implicit function theorem on $\bm{\psi}_-$. 
    
    Without loss of generality, assume that the first $d-1$ columns of $\nabla \bm{\psi}_-$ are linearly independent on a set $B \subset U_- \cup S$ containing $\bm{x}^*$. Note that $B$ is independent of $\tau$. Furthermore, for small enough $\tau$, the trajectories connecting $\bm{x}_k$ with $\bm{x}^*$ and $\hat{\bm{x}}^*$ are contained in $B$. This follows from Lipschitz continuity of the piecewise smooth and discrete solutions:
    \begin{align*}
        \norm{\bm{x}^* - \bm{x}_k} &= \norm{\bm{x}(t^*; \bm{x}_k, t_k) - \bm{x}(t_k; \bm{x}_k, t_k)} \leq L_{\bm{x}} \lvert t^* - t_k \rvert \leq L_{\bm{x}} \tau \\
        \norm{\hat{\bm{x}}^* - \bm{x}_k} &= \norm{\bm{x}^\tau(\hat{t}^*; \bm{x}_k, t_k) - \bm{x}^\tau(t_k; \bm{x}_k, t_k)} \leq L_{\bm{x}^\tau} \lvert \hat{t}^* - t_k \rvert \leq L_{\bm{x}^\tau} \tau 
    \end{align*}
    Thus, we can apply the implicit function theorem to $\bm{\psi}_-$; there exists subsets $(c, x^*_d] \subset \mathbb{R}$ and $(\hat{c}, \hat{x}^*_d] \subset \mathbb{R}$, such that the first $d-1$ coordinates can be written as functions $h_i : (c, x^*_d] \rightarrow \mathbb{R}$ of $x_d$ and $\hat{h}_i : (\hat{c}, \hat{x}^*_d] \rightarrow \mathbb{R}$ of $x_d^\tau$, for $x_d \in (c, x^*_d]$ and $x_d^\tau \in (\hat{c}, \hat{x}^*_d]$, respectively: 
     
    \[
        \bm{x}(t) = \bm{h}(x_d(t)) := \begin{pmatrix}
            h_1(x_d(t)) \\
            \vdots \\
            h_{d-1}(x_d(t)) \\
            x_d(t)
        \end{pmatrix}, \; \bm{x}^\tau(t) = \hat{\bm{h}}(x_d^\tau(t)) := \begin{pmatrix}
            \hat{h}_1(x_d^\tau(t)) \\
            \vdots \\
            \hat{h}_{d-1}(x_d^\tau(t)) \\
            x_d^\tau(t)
        \end{pmatrix}.
    \]
    Let $(\bm{x}_k)_d$ denote the $d^{th}$ component of $\bm{x}_k$. Assume for the moment the following:
    \begin{align}
        x_d^* \text{ is the unique point in } &((\bm{x}_k)_d, x_d^*] \subset (c, x_d^*] \text{ such that } g(\bm{h}(x_d^*)) = 0. \label{lem:uniqueIntersectPWS} \\
        \hat{x}_d^* \text{ is the unique point in } &((\bm{x}_k)_d, \hat{x}_d^*] \subset (\hat{c}, \hat{x}_d^*] \text{ such that } g(\hat{\bm{h}}(\hat{x}_d^*)) = 0. \label{lem:uniqueIntersectDisc}
    \end{align}
    Then we claim that $x_d^* = \hat{x}_d^*$. This follows by contradiction. If $x^*_d > \hat{x}^*_d$, then by uniqueness in $((\bm{x}_k)_d, x_d^*]$ we have that $x_d^* = \hat{x}_d^*$. Similarly, if $x^*_d < \hat{x}^*_d$, then by uniqueness in $((\bm{x}_k)_d, \hat{x}_d^*]$ we have that $x_d^* = \hat{x}_d^*$.
    
    It remains to show \eqref{lem:uniqueIntersectPWS} and \eqref{lem:uniqueIntersectDisc}. First note that $x^*_d$ and $\hat{x}^*_d$ satisfy $g(\bm{h}(x^*_d)) = 0$ and $g(\hat{\bm{h}}(\hat{x}^*_d)) = 0$, respectively. Thus, there is at least one stated point in each respective interval. Lemmas~\ref{lem:ctsMonotonicTime2} and \ref{lem:discMonotonicTime2} from Section~\ref{sec:signAnalysis} shows that for small enough $\tau$, both the piecewise smooth and discrete trajectories stay on one side of the interface. Specifically, for small enough $\tau$, $g(\bm{h}(x_d)) < 0$ for $x_d \in ((\bm{x}_k)_d, x_d^*)$. Similarly, $g(\hat{\bm{h}}(x_d^\tau)) < 0$ for  $x_d^\tau \in ((\bm{x}_k)_d, \hat{x}_d^*)$. Thus, we have shown \eqref{lem:uniqueIntersectPWS} and \eqref{lem:uniqueIntersectDisc}. Thus, $\bm{x}^* = \hat{\bm{x}}^*$.
\end{proof}

\subsection{Proof of Theorem \ref{thm:mainResults}}
\label{sec:mainResults}
We are now in position to prove our main result: 
\[
\norm{\bm{x}(t_{k+1}; \bm{x}_k, t_k) - 
\bm{x}^\tau(t_{k+1}; \bm{x}_k, t_k)} = \mathcal{O}(\tau^p)\,
\]
We note that since $\bm{x}(t^\ast;\bm{x}_{k}, t_{k}) = \bm{x}^\ast$, we can change the starting point of the piecewise smooth trajectory to be at $(\bm{x}^\ast, t^\ast)$, and so $\bm{x}(t_{k+1}; \bm{x}_k, t_k) = \bm{x}(t_{k+1}; \bm{x}^*, t^*)$. Furthermore, due to the inexactness of solving the nonlinear equations~\eqref{eq:to-interface} and~\eqref{eq:interface}, we denote $(\hat{\bm{x}}, \hat{t})$ as a numerical approximation to the roots $(\hat{\bm{x}}^\ast, \hat{t}^\ast)$. As a consequence, $\bm{x}^\tau(\hat{t}; \bm{x}_k, t_k) = \hat{\bm{x}}$ by Algorithm~\ref{alg:trans_scheme} and we can change the starting point of the discrete trajectory to be at $(\hat{\bm{x}}, \hat{t})$. Thus, $\bm{x}^\tau(t_{k+1}; \bm{x}_k, t_k) = \bm{x}^\tau(t_{k+1}; \hat{\bm{x}}, \hat{t})$ and we can rewrite
\begin{equation}
\norm{\bm{x}(t_{k+1}; \bm{x}_k, t_k) - \bm{x}^\tau(t_{k+1}; \bm{x}_k, t_k)} = \norm{\bm{x}(t_{k+1}; \bm{x}^\ast, t^\ast) - \bm{x}^\tau(t_{k+1}; \hat{\bm{x}}, \hat{t})}\, . \label{eq:baseEquation}
\end{equation}
Furthermore, we point out an important geometric aspect of our error analysis. By our hypothesis of Theorem \ref{thm:mainResults}, the conservative schemes on each part of the phase space preserves the $d-1$ conserved quantities. Thus, by Lemma~\ref{lem:geometricLemma}
\begin{equation}
    \norm{\bm{x}^\ast - \hat{\bm{x}}^\ast} < \epsilon.
\end{equation}
Now to prove the main result, we separate into three cases: $t^*=\hat{t}^*$, $t^*>\hat{t}^*$ and $t^*<\hat{t}^*$. In each of these cases, we need statements regarding the monotonicity of $g(\bm{x}(t))$ and $g(\bm{x}^\tau(t))$ near $t^*$ and $\hat{t}^*$ for sufficiently small $t$. Lemmas~\ref{lem:ctsMonotonicTime} through \ref{lem:discMonotonicTime2} from Appendix~\ref{sec:signAnalysis} provide these statements and their proofs. We now proceed to prove each of the three cases.


\subsubsection*{Case $t^\ast = \hat{t}^\ast$} 
By equation~\eqref{eq:baseEquation}, we have by triangle inequality 
\begin{align*}
    \norm{\bm{x}(t_{k+1}; \bm{x}^*, t^*) - \bm{x}^\tau(t_{k+1}; \hat{\bm{x}}, \hat{t})} &\leq \underbrace{\norm{\bm{x}(t_{k+1}; \bm{x}^*, t^*) - \bm{x}^\tau(t_{k+1}; \bm{x}^*, t^*)}}_{\text{(A.1)}} \\
        &\qquad+ \underbrace{\norm{\bm{x}^\tau(t_{k+1}; \bm{x}^*, t^*) - \bm{x}^\tau(t_{k+1}; \hat{\bm{x}}, \hat{t})}}_{\text{(A.2)}}
\end{align*}
We first deal with term (A.1). For sufficiently small $\tau$, $g(\bm{x}(t_{k+1}; \bm{x}^\ast, t^\ast)) > 0$ by Lemma~\ref{lem:ctsMonotonicTime}, which implies that $\bm{x}(t_{k+1}; \bm{x}^*, t^*)$ lies in $U_+$. Similarly, for sufficiently small $\tau$, $g(\bm{x}^\tau(t_{k+1}; \bm{x}^\ast, t^\ast))>0$ by Lemma~\ref{lem:discMonotonicTime}, which implies that $\bm{x}^\tau(t_{k+1}; \bm{x}^\ast, t^\ast)$ lies in $U_+$. Thus, the term (A.1) is $\mathcal{O}(\tau^p)$ by the convergence order of the scheme in the smooth region $U_+$. For term (A.2), Lipschitz continuity of initial data with Lipschitz constant $\hat{L}$ implies that
\begin{align*}
    \norm{\bm{x}^\tau(t_{k+1}; \bm{x}^*, t^*) - \bm{x}^\tau(t_{k+1}; \hat{\bm{x}}, \hat{t})} &\leq \hat{L} \left( \norm{\bm{x}^* - \hat{\bm{x}}} + |t^* - \hat{t}| \right).
\end{align*}
Because $\bm{x}^* = \hat{\bm{x}}^*$ and $t^* = \hat{t}^*$ by Lemma~\ref{lem:geometricLemma}, both of these terms are machine precision. Thus, the main result in the case that $t^* = \hat{t}^*$ is proven.

\subsubsection*{Case $t^\ast > \hat{t}^\ast$} 
Again, by equation~\eqref{eq:baseEquation}, we have by triangle inequality
\begin{align*}
    \norm{\bm{x}(t_{k+1}; \bm{x}^*, t^*) - \bm{x}^\tau(t_{k+1}; \hat{\bm{x}}, \hat{t})} &\leq \underbrace{\norm{\bm{x}(t_{k+1}; \bm{x}^\ast, t^\ast) - \bm{x}^\tau(t_{k+1}; \bm{x}^\ast, t^\ast)}}_{\text{(B.1)}} \\
     &\hspace{3mm}+ \underbrace{\norm{\bm{x}^\tau(t_{k+1}; \bm{x}^\ast, t^\ast) - \bm{x}^\tau(t_{k+1}; \hat{\bm{x}}^\ast, \hat{t}^\ast)}}_{\text{(B.2)}} \\
     &\hspace{3mm}+ \underbrace{\norm{\bm{x}^\tau(t_{k+1}; \hat{\bm{x}}^\ast, \hat{t}^\ast) - \bm{x}^\tau(t_{k+1}; \hat{\bm{x}}, \hat{t})}}_{\text{(B.3)}}\, .
\end{align*}
We now show each of these terms is $\mathcal{O}(\tau^p)$, which would prove the result. For term (B.1), $\bm{x}(t_{k+1}; \bm{x}^\ast, t^\ast)$ lies in $U_+$, since $g(\bm{x}(t_{k+1}; \bm{x}^\ast, t^\ast)) > 0$  for sufficiently small $\tau$, by Lemma~\ref{lem:ctsMonotonicTime}. Moreover, $\bm{x}^\tau(t_{k+1}; \bm{x}^\ast, t^\ast)$ lies in $U_+$, since $g(\bm{x}^\tau(t_{k+1}; \bm{x}^\ast, t^\ast))>0$ for sufficiently small $\tau$, by Lemma~\ref{lem:discMonotonicTime}. The discrete and piecewise smooth solutions used in this term start on $S$ and lie entirely in $U_+$ except at the start and so we can conclude that term (B.1) above is $\mathcal{O}(\tau^p)$.

For term (B.2), since the discrete solution is assumed uniformly Lipschitz in the initial data, with Lipschitz constant $\hat{L}$,
\begin{equation*}
        \norm{\bm{x}^\tau(t_{k+1}; \bm{x}^\ast, t^\ast) - \bm{x}^\tau(t_{k+1}; \hat{\bm{x}}^\ast, \hat{t}^\ast)} \leq \hat{L}( \norm{\bm{x}^\ast - \hat{\bm{x}}^\ast} + |t^\ast - \hat{t}^\ast|)\, .
\end{equation*}
By equation~\eqref{eq:x^*Approx}, this difference reduces to estimating the time error $|t^\ast - \hat{t}^\ast|$ which is discussed below. 

For term (B.3), we use the same Lipschitz continuity argument as in (B.2) and the fact that $(\hat{\bm{x}}, \hat{t})$ can be made to converge to $(\hat{\bm{x}}^\ast, \hat{t}^\ast)$ to an arbitrary precision $\epsilon$.

Now, it remains to show that $\lvert t^\ast - \hat{t}^\ast \rvert = \mathcal{O}(\tau^p)$. To do so, we use Lemma~\ref{lem:ctsTimeError}, with $t=\hat{t}^\ast$
\[
    |\hat{t}^\ast - t^\ast| \leq \frac{M(\hat{t}^\ast-t^\ast)^2 + L_g\norm{\bm{x}(\hat{t}^\ast;\bm{x}_k,t_k) - \bm{x}(t^\ast;\bm{x}_k,t_k)}}{\alpha_S^2},
\]
where $L_g$ is the uniform Lipschitz constant of $g$ and $M$ is as defined in the lemma. Equivalently,
\begin{equation}
    0 \leq M(\hat{t}^\ast-t^\ast)^2 - \alpha_S^2 |\hat{t}^\ast - t^\ast| + L_g\norm{\bm{x}(\hat{t}^\ast;\bm{x}_k,t_k) - \bm{x}(t^\ast;\bm{x}_k,t_k)}\, . \label{eq:t^*>hatT^*}
\end{equation}
To estimate the norm term in \eqref{eq:t^*>hatT^*} note that
\begin{align*}
    \norm{\bm{x}(t^\ast;\bm{x}_k, t_k) - \bm{x}(\hat{t}^\ast;\bm{x}_k, t_k)} &\leq \underbrace{\norm{\bm{x}(\hat{t}^\ast;\bm{x}_k, t_k) - \bm{x}^\tau(\hat{t}^\ast;\bm{x}_k, t_k)}}_{\text{(B.4)}} \\
 &\hspace{1em} + \underbrace{\norm{\bm{x}^\tau(\hat{t}^\ast;\bm{x}_k, t_k) - \bm{x}(t^\ast;\bm{x}_k, t_k)}}_{\text{(B.5)}}\, .
\end{align*}
For term (B.4), by Lemma~\ref{lem:ctsMonotonicTime2} with $t=\hat{t}^\ast$ and for sufficiently small $\tau$, $g(\bm{x}(\hat{t}^\ast)) < 0$ since $t^\ast < \hat{t}^\ast$. So $\bm{x}(\hat{t}^\ast) \in U_-$. Moreover, since $g(\hat{\bm{x}}^\ast) = 0$, $\hat{\bm{x}}^\ast \in S$. Thus, both the piecewise smooth and discrete trajectories lie in $U_- \cup S$, which implies that the term (B.4) is $\mathcal{O}(\tau^p)$. For term (B.5) note that since $\bm{x}^\tau(\hat{t}^\ast) = \hat{\bm{x}}^\ast$, and $\bm{x}(t^\ast) = \bm{x}^\ast$, term (B.5) can be made as small as desired, for example less than $\epsilon$. Thus equation~\eqref{eq:t^*>hatT^*} reduces to
\begin{equation*}
  0 \leq M |t^\ast - \hat{t}^\ast|^2 - \alpha_S^2 |t^\ast - \hat{t}^\ast| + L_g (C \tau^p + \epsilon)\, ,
\end{equation*}
where $C$ is a constant implicit in $\mathcal{O}(\tau^p)$. Then since $\epsilon$ can be made as small as desired, we can absorb it into the $\tau^p$ term. Note also that $|t^\ast - \hat{t}^\ast| \leq \tau$, since both times are contained in the interval $[t_k, t_{k+1}]$. Thus, we can apply Lemma~\ref{lem:quadraticIneq} and ensure that for sufficiently small $\tau$, $|t^\ast - \hat{t}^\ast|$ is in the lower interval of Lemma~\ref{lem:quadraticIneq}. This implies that $|t^\ast - \hat{t}^\ast|$ is $\mathcal{O}(\tau^p)$.

\subsubsection*{Case $t^\ast < \hat{t}^\ast$}
The analysis plays out similarly to the previous case. Once again by equation~\eqref{eq:baseEquation}, we have by triangle inequality 
\begin{align*}    
    \norm{\bm{x}(t_{k+1}; \bm{x}^*, t^*) - \bm{x}^\tau(t_{k+1}; \hat{\bm{x}}, \hat{t})} &\leq \underbrace{\norm{\bm{x}(t_{k+1}; \hat{\bm{x}}^\ast, \hat{t}^\ast) - \bm{x}^\tau(t_{k+1}; \hat{\bm{x}}^\ast, \hat{t}^\ast)}}_{\text{(C.1)}} \\
     &\hspace{3mm}+ \underbrace{\norm{\bm{x}(t_{k+1}; \bm{x}^\ast, t^\ast) - \bm{x}(t_{k+1}; \hat{\bm{x}}^\ast, \hat{t}^\ast)}}_{\text{(C.2)}} \\
     &\hspace{3mm}+ \underbrace{\norm{\bm{x}^\tau(t_{k+1}; \hat{\bm{x}}^\ast, \hat{t}^\ast) - \bm{x}^\tau(t_{k+1}; \hat{\bm{x}}, \hat{t})}}_{\text{(C.3)}}\, .
\end{align*}
We once again show that each of these terms is $\mathcal{O}(\tau^p)$.

The analysis for the terms (C.1), (C.2) and (C.3) are the same as in previous case, with slight modifications. Since $\norm{\bm{x}^\ast - \hat{\bm{x}}^\ast} < \epsilon$ and $|t^\ast - \hat{t}^\ast| < \tau$, then for sufficiently small $\tau$, we have that $\bm{x}(t_{k+1}; \hat{\bm{x}}^\ast, \hat{t}^\ast) \in U_+$ and $\bm{x}^\tau(t_{k+1}; \hat{\bm{x}}^\ast, \hat{t}^\ast) \in U_+$, since $g(\bm{x}(t_{k+1}; \hat{\bm{x}}^\ast, \hat{t}^\ast)) > 0$ and $g(\bm{x}^\tau(t_{k+1}; \hat{\bm{x}}^\ast; \hat{t}^\ast)) > 0$ by Lemma~\ref{lem:ctsMonotonicTime} and Lemma~\ref{lem:discMonotonicTime}, respectively. For term (C.2), we previously assumed that the discrete solution was Lipschitz in the initial data, and now we assume that piecewise smooth solution is Lipschitz in the initial data. Note that we still require that the discrete solution is Lipschitz in the initial data for term (C.3). 

Now, it remains to show that $\lvert t^\ast - \hat{t}^\ast \rvert = \mathcal{O}(\tau^p)$. To do so, we use Lemma~\ref{lem:discTimeError} with $t=t^\ast$ and for sufficiently small $\tau$. Thus~\eqref{eq:discTimeError} yields
\begin{equation}
    0 \leq \hat{M}(t^\ast - \hat{t}^\ast)^2 - \hat{\alpha_S}^2 |t^\ast - \hat{t}^\ast| + L_g\norm{\bm{x}^\tau(t^\ast;\bm{x}_k,t_k) - \bm{x}^\tau(\hat{t}^\ast;\bm{x}_k,t_k)}\, . \label{eq:t^*<hatT^*}
\end{equation}
We estimate the norm term in \eqref{eq:t^*<hatT^*} as
\begin{align*}
    \norm{\bm{x}^\tau(t^\ast;\bm{x}_k, t_k) - \bm{x}^\tau(\hat{t}^\ast;\bm{x}_k, t_k)} &\leq \underbrace{\norm{\bm{x}^\tau(t^\ast;\bm{x}_k, t_k) - \bm{x}(t^\ast;\bm{x}_k, t_k)}}_{\text{(C.4)}} \\
     &+ \underbrace{\norm{\bm{x}(t^\ast;\bm{x}_k, t_k) - \bm{x}^\tau(\hat{t}^\ast;\bm{x}_k, t_k)}}_{\text{(C.5)}}\, .
\end{align*}
For term (C.4), by Lemma~\ref{lem:discMonotonicTime2} with $t=t^\ast$ and sufficiently small $\tau$, since $t^\ast < \hat{t}^\ast$, $g(\bm{x}^\tau(t^\ast)) < 0$, and so $\bm{x}^\tau(t^\ast) \in U_-$. Since $g(\bm{x}^\ast) = 0$, $\bm{x}^\ast \in S$ and thus both the discrete and piecewise smooth trajectories lie in $U_-\cup S$ which implies that the term (C.4) is $\mathcal{O}(\tau^p)$. For term (C.5), $\bm{x}^\tau(\hat{t}^\ast) = \hat{\bm{x}}^\ast$, and $\bm{x}(t^\ast) = \bm{x}^\ast$ and so this term can be made smaller than any $\epsilon >0$.
Thus, equation~\eqref{eq:t^*>hatT^*} with $t=t^\ast$ reduces to
\begin{equation*}
  0 \leq \hat{M} |t^\ast - \hat{t}^\ast|^2 - \hat{\alpha}_S^2 |t^\ast - \hat{t}^\ast| + L_g(C \tau^p + \epsilon)\, .
\end{equation*}
As in the previous case, because $\epsilon$ can be made arbitrarily small, we can absorb it in the $\tau^p$ term. Just as before, $|t^\ast - \hat{t}^\ast| \leq \tau$, since both times are contained in the interval $[t_k, t_{k+1}]$. Thus, we can apply Lemma~\ref{lem:quadraticIneq} and ensure that for sufficiently small $\tau$, $|t^\ast - \hat{t}^\ast|$ is in the lower interval of Lemma~\ref{lem:quadraticIneq}. This implies that $|t^\ast - \hat{t}^\ast|$ is $\mathcal{O}(\tau^p)$, completing the proof.

\section{Numerical results}
\label{sec:numerical}

We computed the error before and after crossing the interface for two systems - the undamped harmonic oscillator and an elliptic curve system, both with a discontinuity in a parameter. Some trajectories are shown in Figure~\ref{fig:portraits}. We examined the effect of adding error artificially to the computed time of intersection. (The solution trajectory shape is known up to machine precision, the only error is in time.) 
\begin{figure}[h!]
  \centering
  \includegraphics[width=.45\linewidth]{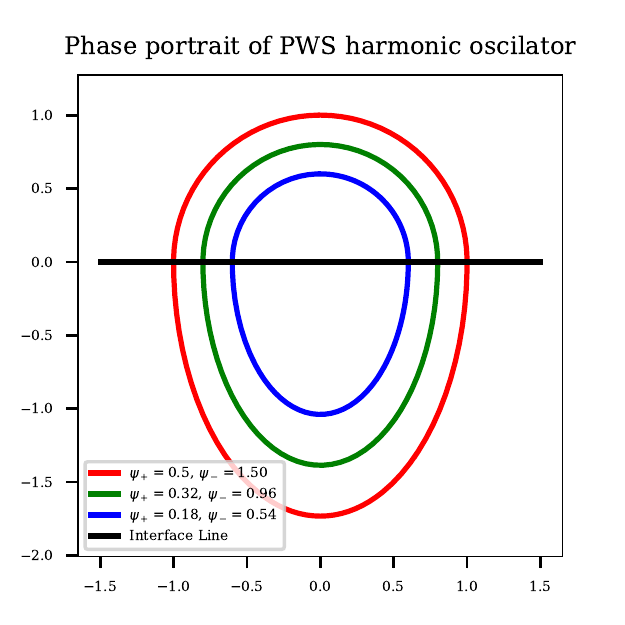}
  \includegraphics[width=0.45\linewidth]{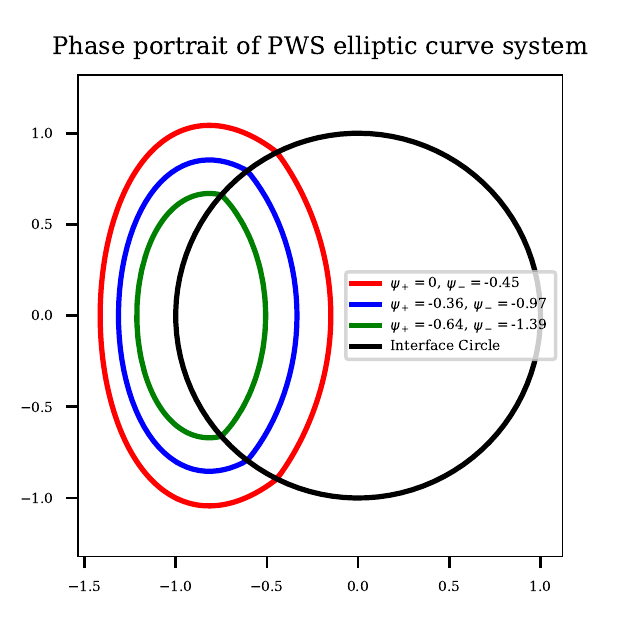}
  \caption{Left plot shows trajectories of a harmonic oscillator with different spring constant values across the interface $y=0$. Right plot shows trajectories for the elliptic curve example with the interface being a circle of radius 1 centered at origin with different parameter values inside and outside the circle.}
  \label{fig:portraits}
\end{figure}

This allows for a demonstration of Theorem~\ref{thm:mainResults}. If the error introduced into the intersection time value is of an order larger than the global truncation error of the method then the order of the method deteriorates after crossing the interface and if the error is smaller then it does not, as we have proved.

\subsection{Harmonic oscillator system}\label{subsec:harmonic}
Consider the undamped simple harmonic oscillator: $\dot{x} = y$ and $\dot{y} = -\omega^2 x$ where $\omega > 0$ is the natural frequency of the oscillator. Recall that there is a conserved quantity called the energy, $\psi_\pm = (\omega_\pm^2 x^2 + y^2)/2$. We will make this system piecewise smooth by using two different values for $\omega$ in two regions of the phase space. Specifically, in the notation introduced in 
Section~\ref{subsec:background}, $U = \mathbb{R}^2$, $g(x,y) = y$ and the switching surface $S$ is the horizontal line $y=0$. Thus $U_{\pm} = \{(x,y) \in \mathbb{R}^2 \; \vert\; y \gtrless 0 \}$ and the PWS vector field consists of $\bm{f}_-(t, x, y) = (y, -3x)$ and $\bm{f}_+(t,x, y) = (y, -x)$.

The DMM integrator used for the harmonic oscillator is the implicit midpoint method, which has a global truncation error of $\mathcal{O}(\tau^2)$.

In the proof of the main theorem, an important step was showing that $\vert t^\ast - \hat{t}^\ast \vert$ is $\mathcal{O}(\tau^p)$. This was shown using the quadratic inequality of Lemma~\ref{lem:quadraticIneq}. To demonstrate this numerically, since we don't have access to the exact crossing time $\hat{t}^\ast$ of the discrete trajectory, we use $\hat{t}$ as a proxy ($\hat{t}$ is the numerical approximation of $\hat{t}^\ast$ obtained by solving the nonlinear equation~\eqref{eq:to-interface} and can be made as close to $\hat{t}^\ast$ as the nonlinear solver allows.) The crossing time $t^\ast$ of the trajectory of the PWS system is known from the analytical solution. Figure~\ref{fig:harmonic_timeError} illustrates this important $\mathcal{O}(\tau^p)$ result where $p=2$ in this example. Notice that the $\vert t^\ast - \hat{t} \vert$ remains $\mathcal{O}(\tau^2)$ after 10, 20 and 30 transitions even though there is accumulation of error with increasing transitions.
\begin{figure}[h!]
  \centering
    \includegraphics{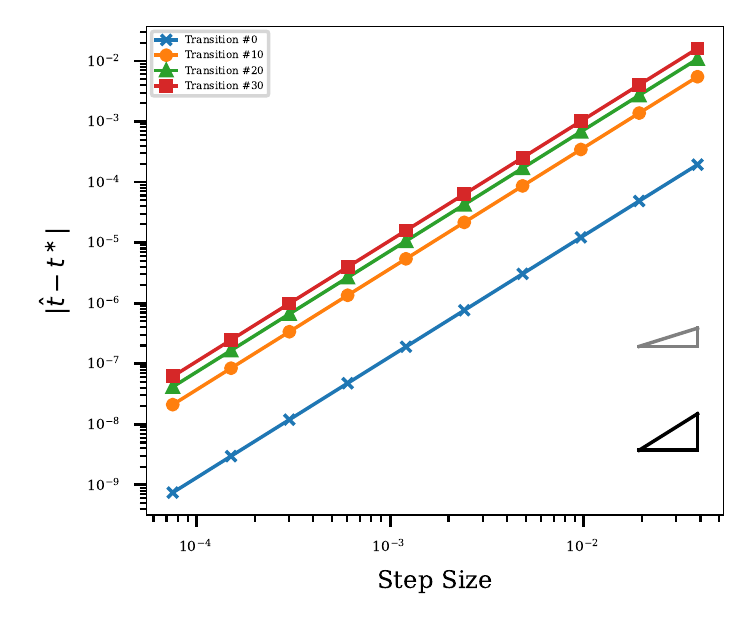}
  \label{fig:harmonic_timeError}
  \caption{The crossing time error $\vert t^\ast-\hat{t} \vert$ between the PWS and discrete trajectory for harmonic oscillator. Here the numerical approximation $\hat{t}$ is used as a proxy for $\hat{t}^\ast$ since the two can be made as close as desired. As proved in Section~\ref{sec:mainResults} the order of the time error is $\mathcal{O}(\tau^p)$. Here $p=2$. }
\end{figure}

\begin{figure}[h!]
    \centering
    \includegraphics{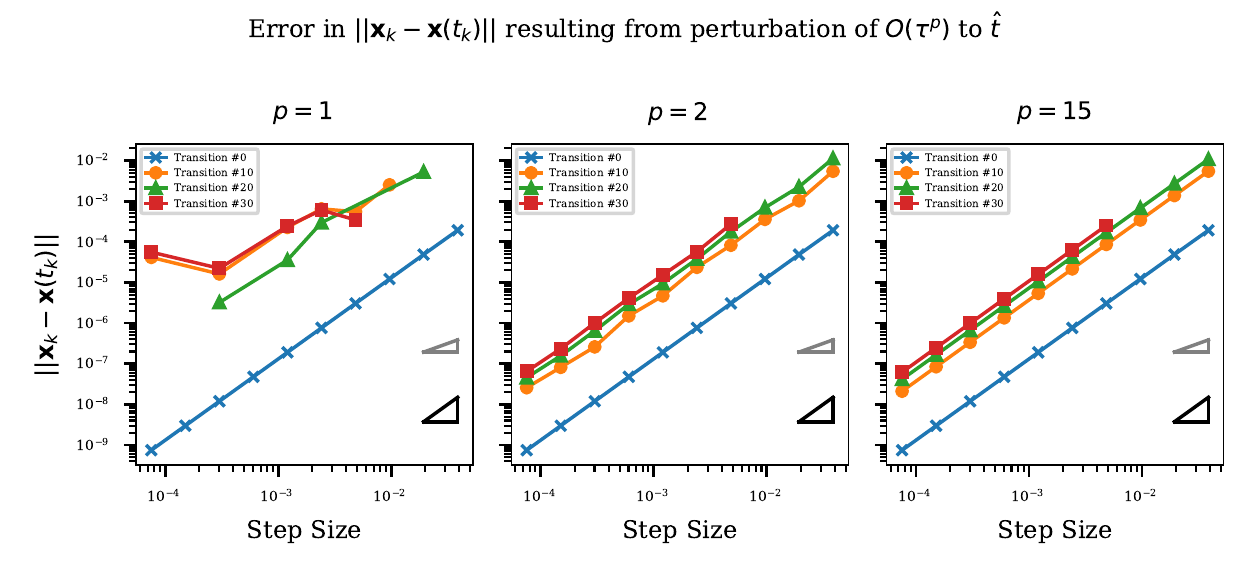}
    \caption{Error in the solution for the PWS harmonic oscillator before and after crossing interface multiple times. A perturbation of $\mathcal{O}(\tau^p)$ is artificially added to $\hat{t}$, the time that is computed to be the interface crossing time. In the first row $p=1, 2$ and $15$ were used (the last one being effectively no error).}
    \label{fig:harmonic_transition}
\end{figure}

\begin{figure}[h!]
    \centering
    \includegraphics{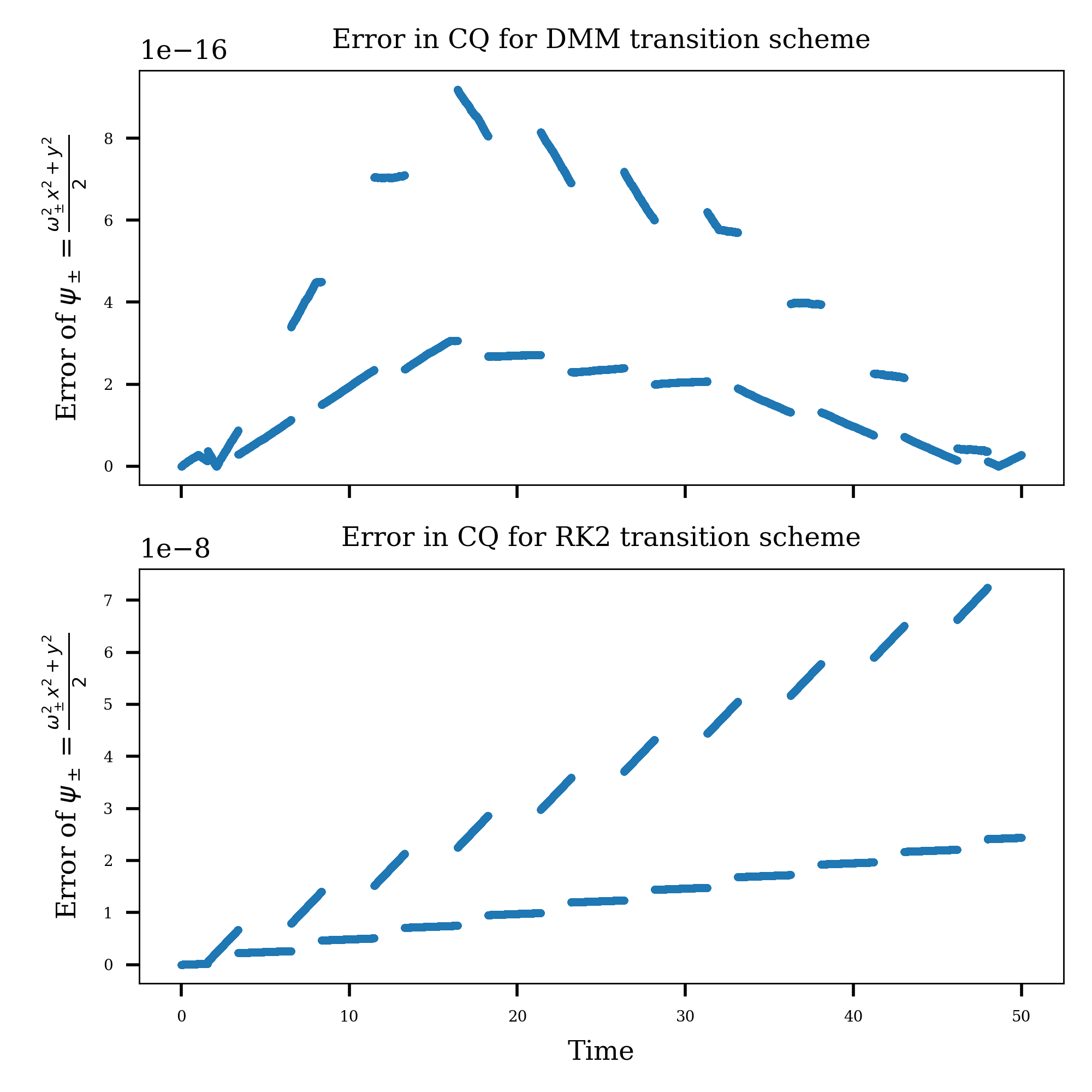}
    \caption{Comparison of our DMM transition scheme with the explicit midpoint method (which is equivalent to a version of RK2) in computing conserved quantity (total energy) of a PWS harmonic oscillator. A switching surface of $y=0$ was used, with $\omega^2=1$ for $y>0$ and $\omega^2=3$ for $y<0$. The error in energy as computed by the DMM method (which reduces to the midpoint method in this case) is in the top plot and as computed by RK2 is in the bottom plot. An initial condition of $[1,1]$ and the same time step size were used for both methods.}
    \label{fig:CQ_harmonic}
\end{figure}

Figure~\ref{fig:harmonic_transition} shows the behavior of error in the PWS system before and after crossing the interface when a perturbation is introduced into the computation of $\hat{t}$. The size of the perturbation added is of order $\mathcal{O}(\tau^p)$ for $p=1,2, 15$. While the $p=1$ perturbation is larger than the truncation error of the method, the $p=2$ perturbation is of the same order and the $p=15$ perturbation introduces effectively no error to $\hat{t}$. The last plot in Figure~\ref{fig:harmonic_transition} confirms that when no perturbation is added to $\hat{t}$, the order of accuracy before and after the transition stays the same, even after the interface has been crossed 30 times. In contrast, the first plot in Figure~\ref{fig:harmonic_transition} shows that if intersection time is not computed accurately enough, the accuracy degrades to $\mathcal{O}(\tau)$ as predicted by Mannshardt~\cite{Mannshardt1978}.  

Figure~\ref{fig:CQ_harmonic} shows the ability of our DMM transition scheme to accurately reproduce the conserved quantities on the two sides of the interface even after multiple transitions. As expected, the error in the energy is close to machine precision for our method, whereas the RK2 integration shows an error of about $10^{-8}$ and a drift in the energy over time. A time step size of $10^{-3}$ was used for both integrators.

\subsection{Elliptic curve system}\label{subsec:elliptic}
Our second example is the elliptic curve example from~\cite{WaNa2018}, which is $\dot{x} = 2y$ and $\dot{y} = 3x^2 + a_\pm$ with the difference in the parameter $a_\pm$ introducing the discontinuity in the system. The conserved quantities for this system are $\psi_\pm = y^2 - x^3 - a_\pm x$, which defines an elliptic curve motivating the name of the system. We chose the phase space to be divided into two regions by a circular interface of radius 1 centered at the origin.  The parameter value inside the circle is $a_-=-3$ and outside is $a_+=-2$. Thus $U = \mathbb{R}^2$, $g(x,y) = x^2 + y^2 - 1$ and the switching surface $S$ is the circle $g(x,y) = x^2 + y^2 - 1 = 0$. The two parts of the phase space are $U_{\pm} = \{(x,y) \in \mathbb{R}^2 \; \vert\; x^2 + y^2 \gtrless 1 \}$ and the PWS vector field consists of $\bm{f}_-(t, x, y) = (2y, 3x^2-3)$ and $\bm{f}_+(t,x, y) = (2y, 3x^2-2)$.  

 \begin{figure}[h!]
  \centering
    \includegraphics{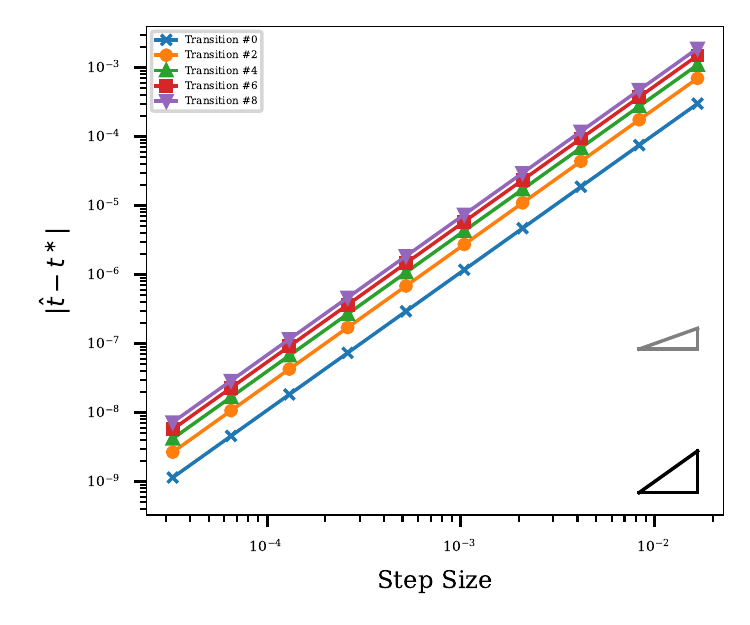}
  \caption{The crossing time error $\vert t^\ast-\hat{t} \vert$ between the PWS and discrete trajectory for elliptic curve system. Because RK4 is being used as a reference solution, the value of $\hat{t}$ found by RK4 is used as a substitute for $t^\ast$.}
  \label{fig:elliptic_time}
\end{figure}

\begin{figure}[h!]
  \centering
  \includegraphics{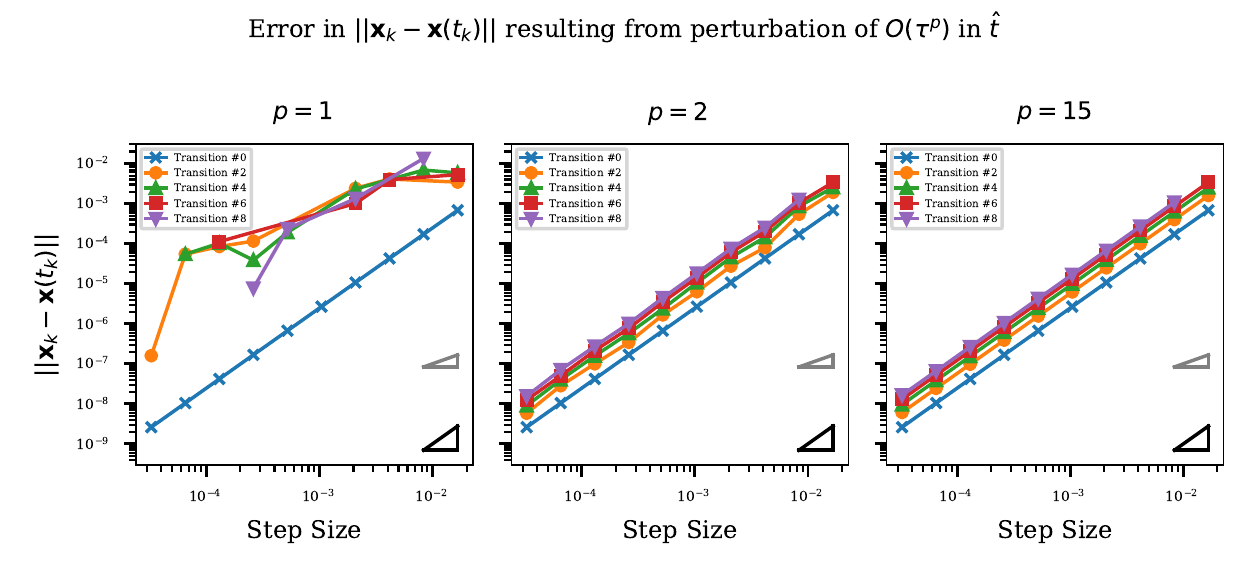}
  \caption{Error in the solution for the PWS elliptic curve system before and after crossing interface multiple times. See caption of Figure~\ref{fig:harmonic_transition} and the text for explanation.}
  \label{fig:elliptic_transition}
\end{figure}

\begin{figure}[h!]
    \centering
    \includegraphics{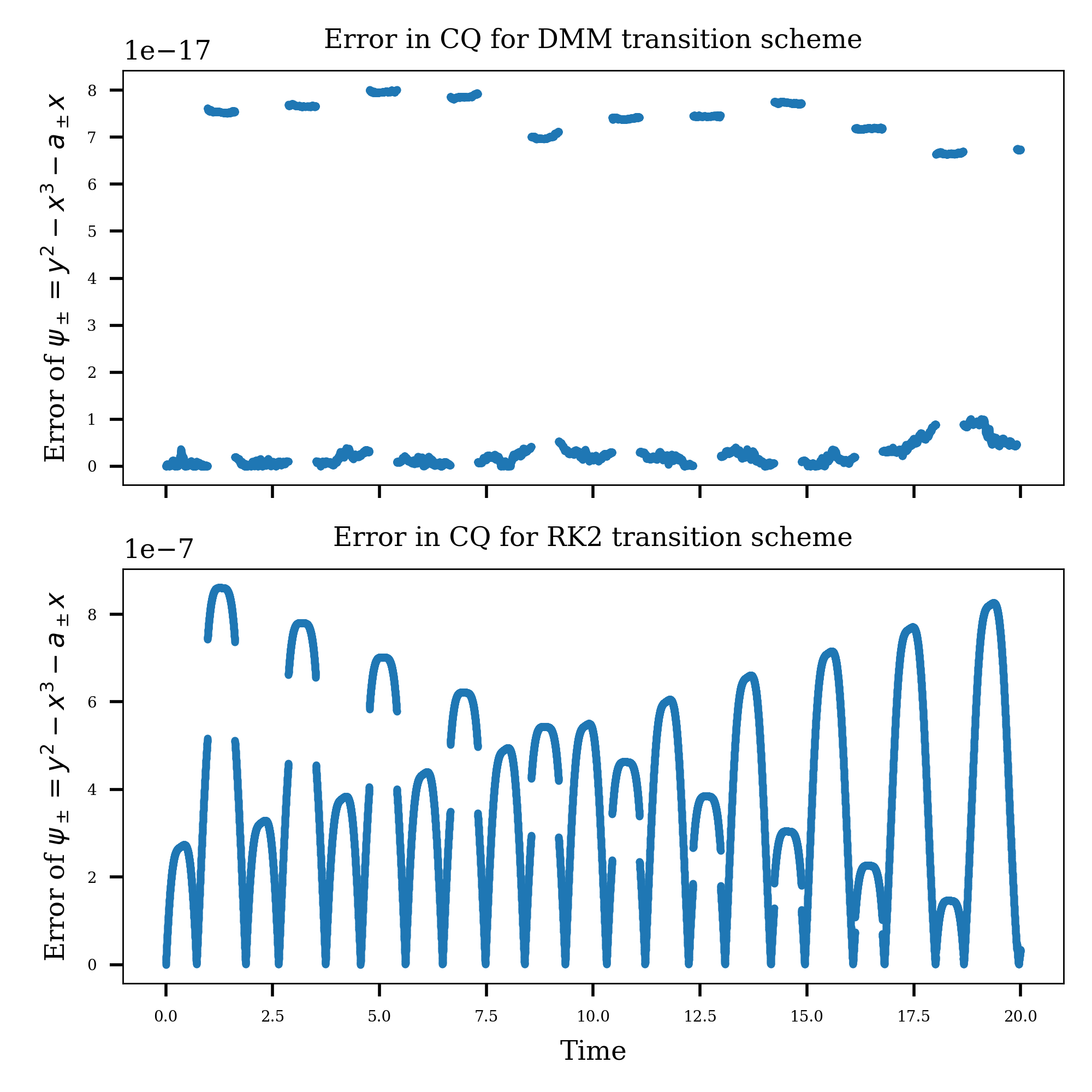}
    \caption{Comparison of our DMM transition scheme with RK2 transition scheme in computing conserved quantity (CQ) of a PWS elliptic curve system. A switching surface of $g(x,y) = y^2 + x^2 - 1 = 0$ was used, with $a=-2$ for $g>0$ and $a=-3$ for $g<0$. The error in CQ as computed by the DMM method is the top plot and as computed by RK2 is the bottom plot. An initial condition of $[-1,-1]$ and a step size $10^{-3}$ was used for both methods.}
    \label{fig:CQ_elliptic}
\end{figure}

The DMM integrator for the elliptic curve system was a symmetric\footnote{A one-step integrator is called symmetric if after exchanging quantities defined at time $t_k$ and $t_{k+1}$ and replacing $\tau$ by $-\tau$ yields the same integrator. It is known that a symmetric one step method has even order of accuracy. See \cite{HaLuWa2006} for more details.} method and as expected, it was empirically observed to have global truncation error of $\mathcal{O}(\tau^2)$. The error was computed by comparing with a reference solution computed by a 4th order Runge-Kutta fixed time step integrator using a time step of approximately $\tau = 1.6 \times 10^{-5}$. The numerical experiments setup is analogous to the PWS Harmonic oscillator. Both figures~\ref{fig:elliptic_time} and \ref{fig:elliptic_transition} show a behavior similar to that observed in the simpler harmonic oscillator example in this nonlinear example, which has a more interesting curved interface shape as compared to the straight line in the previous example and which has a cubic polynomial conserved quantity.

Figure~\ref{fig:CQ_elliptic} shows the ability of our DMM transition scheme to accurately reproduce the conserved quantities on the two sides of the interface even after multiple transitions. As expected, the error in the energy is close to machine precision for our method, whereas the RK2 integration shows an error of about $10^{-7}$ and a drift in the energy over time. A time step size of $10^{-3}$ was used for both integrators.

\section*{Acknowledgments}
The research of ANH was supported in part by KAUST Office of Sponsored Research under Award No. URF/1/3723-01-01 and by the DARPA Defense Sciences Office under Award No. HR0011-20-20019. The research of ATSW was supported by the NSERC Discovery Grant program and NSERC Launch Supplement. The research of NW was supported by KAUST Office of Sponsored Research under Award No. URF/1/3723-01-01. ANH and ATSW would like to thank the hospitality of the Isaac Newton Institute, Cambridge University during the Workshop on Geometry, Compatibility and Structure-Preservation in Computational Differential Equations during which part of this research was done.

\bibliographystyle{siamplain}
\bibliography{pws}

\appendix

\section{Existence of \texorpdfstring{$t^*, \vect{x}^*$}{t-star, x-star} and
  \texorpdfstring{$\hat{t}^*, \vect{\hat{x}}^*, \vect{x}_{k+1}$}{t-hat-star, x-hat-star, x k+1}}
\label{app:existence_intersection_points}
In order for the transition scheme of Algorithm~\ref{alg:trans_scheme} to be well-defined, we have implicitly assumed, upon arriving at $\bm{x}_k$, the existence of a transition time and point $t^*,\bm{x}^*$, and as well as the existence of a discrete transition time and point $\hat{t}^*, \bm{\hat{x}}^*$ and the subsequent discrete solution $\bm{x}_{k+1}$. In this section, we prove their existence and show the well-posedness of equations \eqref{eq:to-interface}-\eqref{eq:next}. First, we will show results on the existence of $t^*, \bm{x}^*$. Often in the following analysis, it will be convenient to define the sign function $h(t):=g(\bm{x}(t))$ to indicate phase space regions the trajectory $\bm{x}(t)$ lies in at time $t$.
\begin{lemma}[Sufficient condition for existence of $t^*, \bm{x}^*$] \label{lem:suff_cond_transition}
Let $g \in C(U\rightarrow \mathbb{R})$, $\bm{x}\in C(I\rightarrow U)$ and define $h:=g\circ {\bm x}$. If there are $a<b$ in $I$ such that $h(a)<0<h(b)$, then there exists $t^*\in(a,b)\subseteq I$ and $\bm{x}^*:=\bm x(t^*)$ satisfying $g(\bm{x}^*) = h(t^*)=0$.
\end{lemma}
\begin{proof}This follows directly from the hypothesis and the intermediate value theorem applied to $h(t)$.
\end{proof}
In other words, there is a transition time $t^*$ and point $\bm x^*$ provided $g$ changes sign on some continuous trajectory $\bm{x}(t)$. A partial converse can also be shown using the transversality condition \eqref{eq:transverseA} and for solution $\bm{x}$ satisfying \ref{hyp:H1}-\ref{hyp:H3}.
\begin{lemma}[Partial converse of Lemma \ref{lem:suff_cond_transition}]
Let $[a,b]\subset I$ and $g\in C^1(U\rightarrow \mathbb{R})$ satisfying the transversality condition \eqref{eq:transverseA}. Further suppose there exists only one $t^*\in (a,b)$ with $\bm{x}^*:=\bm{x}(t^*)$ such that $0=h(t^*)=g(\bm{x}^*)$ and that $\bm{x}\in C(I\rightarrow U)$ is a solution of \eqref{eq:typeA} satisfying \ref{hyp:H1}-\ref{hyp:H3}.  Then there exists a $\delta>0$ such that for all $\alpha\in(t^*-\delta,t^*) \subset [a,b]$ and $\beta\in(t^*,t^*+\delta)\subset I$, $h(\alpha)<0<h(\beta)$.  \label{lem:part_conv_transition}
\end{lemma}
\begin{proof} First, we note that $h\in C(I\rightarrow \mathbb{R})$, since $g\in C(U\rightarrow \mathbb{R})$ and $\bm{x}\in C(I\rightarrow U)$. Note, $h'$ is not continuous on $I$ since \begin{align*}
h'(t) =
\begin{cases}
(\nabla g\cdot \bm{f}_-)(t,\bm{x}(t)), & t\in [a,t^*), \\
(\nabla g\cdot \bm{f}_+)(t,\bm{x}(t)), & t\in (t^*,b].
\end{cases}
\end{align*}
However from \ref{hyp:H2}, the one sided limits of $h'$ as $t\rightarrow t^*$ exists. That is, $\lim_{t\uparrow t^*} h'(t^*) = (\nabla g\cdot \bm{f}_-)(t^*,\bm{x}^*)$ and $\lim_{t\downarrow t^*} h'(t^*) = (\nabla g\cdot \bm{f}_+)(t^*,\bm{x}^*)$. In order words, we can extend $h'$ separately so that $h'\in C^1([a,t^*]\rightarrow \mathbb{R})$ and $h'\in C^1([t^*,b]\rightarrow \mathbb{R})$. Since $h(t^*)=0$, then by the mean value theorem applied to $h$ on $[a,t^*]$ and $[t^*,b]$, there exists $\xi_-(t)\in (t, t^*)$ for all $t\in (a,t^*)$ satisfying $h(t) = h'(\xi_-(t))(t-t^*)$ and there exists $\xi_+(t)\in (t^*,t)$ for all $t\in (t^*,b)$ satisfying $h(t) = h'(\xi_+(t))(t-t^*)$. Combining the transversality condition \eqref{eq:transverseA} and the continuity of $\nabla g, \bm{f}_{\pm}, \bm{x}$ from the left and right of $t^*$ separately, there exists $\delta>0$ such that $h'(\eta)>0$ for all $\eta \in (t^*-\delta,t^*)\bigcup(t^*,t^*+\delta)$. In other words, since  for all $\alpha\in(t^*-\delta,t^*)$, $\xi_-(\alpha) \in (\alpha,t^*) \subset (t^*-\delta,t^*)$ and so $h'(\xi_-(\alpha))>0$, which implies $h(\alpha)=h'(\xi_-(\alpha))(\alpha-t^*)<0$ for all $\alpha\in(t^*-\delta,t^*)$. Similarly, since for all $\beta\in(t^*,t^*+\delta)$, $\xi_+(\beta) \in (t^*,\beta) \subset (t^*,t^*+\delta)$ and so $h'(\xi_+(\beta))>0$, which implies $  
h(\beta)=h'(\xi_-(\beta))(\beta-t^*)>0$ for all $\beta\in(t^*,t^*+\delta)$.
\end{proof}
Having discussed the existence of $t^*, \bm{x}^*$, we now show the existence of their discrete counterparts $\hat{t}^*, \hat{\bm{x}}^*$, which satisfy equations \eqref{eq:to-interface}-\eqref{eq:interface}, and $\bm{x}_{k+1}$, which satisfies \eqref{eq:next}. For this, we will assume there is a transition time $t^*\in [t_k, t_{k+1}]$. 
and we will need the following lemma on the existence of a fixed point of the map $T_{\tau_-}$ defined by $T_{\tau_-}(\bm{x}) := \bm{x}_k + \tau_- \bm{f}_-^\tau(t_k, \bm{x}_k, t, \bm{x})$, for any fixed $\tau_- := t-t_k\in [0,\tau]$.
Note that equation \eqref{eq:to-interface} is equivalent to a fixed point $\bm{\hat{x}} = T_{\tau_-}(\bm{\hat{x}})$ for some fixed $\tau_-$.
\begin{lemma}
Let $\delta > 0$ and ${\bm f}_-^\tau(t_k,\bm{x}_k,\cdot,\cdot) \in Lip([t_k,t_{k+1}]\times B_\delta(\bm{x}_k))\rightarrow \mathbb{R}^d$. Let $L_1(\delta)>0$ be the Lipschitz constant so that for $t,t'\in [t_k,t_{k+1}]$ and $\bm x\in B_\delta(\bm{x}_k)$,
 \[
    \norm{f_-^\tau(t_k, \bm{x}_k, t, \bm{x}) - f_-^\tau(t_k, \bm{x}_k, t', \bm{x})} \leq L_1  \vert t - t' \vert\, ,
  \]
and $L_2(\delta)>0$ as the Lipschitz constant so that for $t\in [t_k,t_{k+1}]$ and $\bm x,\bm y\in B_\delta(\bm{x}_k)$,
  \[
    \norm{f_-^\tau(t_k, \bm{x}_k, t, \bm{x}) - f_-^\tau(t_k, \bm{x}_k, t, \bm{y})} \leq L_2 \norm{\bm{x} - \bm{y}}\,.
  \] Also, define $M = \displaystyle\max_{(t,\bm{x}) \in [t_k,t_{k+1}] \times B_\delta(\bm{x}_k)} \norm{f_-^\tau(t_k, \bm{x}_k, t, \bm{x})}$.
  For any $0<\epsilon<\frac{1}{L_2}$ and $\tau_-^\ast := \min \left\{ \frac{\delta}{M}, \frac{1}{L_2}-\epsilon \right\}$, then the following holds:
  \begin{enumerate}[label=(\roman*)]
  \item For any fixed $\tau_- \in [0,\tau_-^\ast]$, the map $T_{\tau_-} : B_\delta(\bm{x}_k) \rightarrow B_\delta(\bm{x}_k)$ is a contraction map and has a unique fixed point $\bm{\hat{x}}(\tau_-)$.
  \item The discrete solution $\bm{\hat{x}} : [0, \tau_-^\ast] \rightarrow B_\delta(\bm{x}_k)$ is Lipschitz continuous in $\tau_{-}$ with its Lipschitz constant depending on $\delta, \epsilon$.
  \end{enumerate}\label{lem:wellposed-lemma1}
\end{lemma}
\begin{proof}
  It suffices to show that $T_{\tau_-}$ satisfies the hypothesis of the Banach Fixed Point Theorem. First, since
$ \norm{T_{\tau_-}(\bm{x}) - \bm{x}_k} \leq \tau_- M < \delta$ for $\tau_- \leq \tau_-^\ast$, the map $T_{\tau_-} : B_\delta(\bm{x}_k) \rightarrow B_\delta(\bm{x}_k)$ is well-defined. Moreover, since $\tau_- L_2 \leq \tau_-^* L_2\leq 1- \epsilon L_2< 1$, for any $\bm{x}, \bm{y} \in B_\delta(\bm{x}_k)$,
  \begin{align*}
    \norm{T_{\tau_-}(\bm{x}) - T_{\tau_-}(\bm{y})} &= \tau_- \norm{f_-^\tau(t_k, \bm{x}_k, t, \bm{x}) - f_-^\tau(t_k, \bm{x}_k, t, \bm{y})} \\ &\leq \tau_- L_2 \norm{\bm{x} - \bm{y}} < \norm{\bm{x} - \bm{y}}.
  \end{align*} Thus, for any fixed $\tau_- \in [0,\tau_-^\ast]$, $T_{\tau_-} : B_\delta(\bm{x}_k) \rightarrow B_\delta(\bm{x}_k)$ is a contraction map and has a unique fixed point $\bm{\hat{x}}(\tau_-)\in B_\delta(\bm{x}_k)$. To show Lipschitz continuity of $\bm{\hat{x}}(\tau_-)$ in $[0, \tau_-^\ast]$, note that for any $\tau_-$, $\tau'_- \in [0, \tau_-^\ast]$, with $\bm{\hat{x}}' := \bm{\hat{x}}(\tau'_-)$
\begin{align*}
  &\norm{\bm{\hat{x}}(\tau_-) - \bm{\hat{x}}(\tau_-')} = \norm{T_{\tau_-}(\bm{\hat{x}}) - T_{\tau_-'}(\bm{\hat{x}}')} = \norm{\tau_- f_-^\tau(t_k, \bm{x}_k, t, \bm{\hat{x}}) - \tau_-' f_-^\tau(t_k, \bm{x}_k, t', \bm{\hat{x}}')} \\
                                        ~  &\leq \tau_- \norm{f_-^\tau(t_k, \bm{x}_k, t, \bm{\hat{x}}) - f_-^\tau(t_k, \bm{x}_k, t', \bm{\hat{x}})} + \tau_-' \norm{f_-^\tau(t_k, \bm{x}_k, t', \bm{\hat{x}}) - f_-^\tau(t_k, \bm{x}_k, t', \bm{\hat{x}}')}\\
                                       &\quad+|\tau_- - \tau_-'| \norm{f_-^\tau(t_k, \bm{x}_k, t', \bm{\hat{x}})}\\
                                        ~  &\leq \tau_- L_1 |t - t'|  + \tau_-' L_2 \vert \bm{\hat{x}} - \bm{\hat{x}}'\vert\, + M |\tau_- - \tau_-'|.\\
&\Rightarrow (1-\tau_-'L_2)\norm{\bm{\hat{x}}(\tau_-) - \bm{\hat{x}}(\tau_-')} \leq (M+\tau_-L_1)|\tau_- - \tau_-'|
\end{align*}
Since $\tau_-'\leq \tau_-^*\leq \frac{1}{L_2}-\epsilon$, then $0<\epsilon L_2\leq 1-\tau_-^*L_2\leq 1-\tau_-'L_2$, which implies
\[
  \norm{ \bm{\hat{x}}(\tau_-) - \bm{\hat{x}}(\tau_-')  } \le \frac{M+\tau_-' L_1}{1-\tau_-'L_2}|\tau_- - \tau_-'|\ \leq \frac{M+\tau_-^* L_1}{\epsilon L_2}|\tau_- - \tau_-'|\, .
\]
\end{proof}
Similarly, if $(\hat{\tau}^*,\hat{\bm{x}}^*)$ exists, then the following lemma can be established for equation \eqref{eq:next} with the map $T_{\tau_+}$ defined by $T_{\tau_+}(\bm{x}) := \bm{\hat{x}}^*+ \tau_+ \bm{f}_+^\tau(\hat{\tau}^*,\hat{\bm{x}}^*, t, \bm{x})$, for any fixed $\tau_+ := t-\hat{t}^*\in [0,\tau]$. We omit the proof, since it is nearly identical to Lemma \ref{lem:wellposed-lemma1}.
\begin{lemma}
Let $\delta > 0$ and ${\bm f}_+^\tau(\hat{\tau}^*,\hat{\bm{x}}^*,\cdot,\cdot) \in Lip([\hat{t}^*,t_{k+1}]\times B_\delta(\hat{\bm{x}}^*))\rightarrow \mathbb{R}^d$. Let $L_1(\delta)>0$ as the Lipschitz constant so that for $t,t'\in [\hat{t}^*,t_{k+1}]$ and $\bm x\in B_\delta(\hat{\bm{x}}^*)$,
 \[
    \norm{f_+^\tau(\hat{\tau}^*,\hat{\bm{x}}^*, t, \bm{x}) - f_+^\tau(\hat{\tau}^*,\hat{\bm{x}}^*, t', \bm{x})} \leq L_1  \vert t - t' \vert\, ,
  \]
and $L_2(\delta)>0$ as the Lipschitz constant so that for $t\in [t_k,t_{k+1}]$ and $\bm x,\bm y\in B_\delta(\hat{\bm{x}}^*)$,
  \[
    \norm{f_+^\tau(\hat{\tau}^*,\hat{\bm{x}}^*, t, \bm{x}) - f_+^\tau(\hat{\tau}^*,\hat{\bm{x}}^*, t, \bm{y})} \leq L_2 \norm{\bm{x} - \bm{y}}\,.
  \] Also, define $M = \displaystyle\max_{(t,\bm{x}) \in [t_k,t_{k+1}] \times B_\delta(\hat{\bm{x}}^*)} \norm{f_+^\tau(\hat{\tau}^*,\hat{\bm{x}}^*, t, \bm{x})}$.
  For any $0<\epsilon<\frac{1}{L_2}$ and $\tau_+^\ast := \min \left\{ \frac{\delta}{M}, \frac{1}{L_2}-\epsilon \right\}$, then the following holds:
  \begin{enumerate}[label=(\roman*)]
  \item For any fixed $\tau_+ \in [0,\tau_+^\ast]$, the map $T_{\tau_+} : B_\delta(\hat{\bm{x}}^*) \rightarrow B_\delta(\hat{\bm{x}}^*)$ is a contraction map and has a unique fixed point $\bm{\hat{x}}(\tau_+)$.
  \item The discrete solution $\bm{\hat{x}} : [0, \tau_+^\ast] \rightarrow B_\delta(\hat{\bm{x}}^*)$ is Lipschitz continuous in $\tau_{+}$ with its Lipschitz constant depending on $\delta, \epsilon$.
  \end{enumerate}\label{lem:wellposed-lemma2}
\end{lemma}
Now we are ready to show the existence of $\hat{t}^*,\hat{\bm{x}}^*$ and $\bm{x}_{k+1}$ for equations \eqref{eq:to-interface}-\eqref{eq:next}. From the transition scheme as described in Algorithm \ref{alg:trans_scheme}, we would only seek $\hat{t}^*,\hat{\bm{x}}^*$ if the proposed discrete solution $\tilde{\bm{x}}_{k+1}$ have a different sign $g(\tilde{\bm{x}}_{k+1})$ than $g(\bm{x}_k)$. Thus for concreteness and without loss of generality, we will assume that $g(\bm{x}_k)<0$ and $g(\tilde{\bm{x}}_{k+1})>0$.
\begin{proposition}[Existence of $\hat{t}^*,\hat{\bm{x}}^*$ and $\bm{x}_{k+1}$]
Fix a $\delta>0$ and suppose for such $\delta$, the hypothesis of Lemma \ref{lem:wellposed-lemma1} and \ref{lem:wellposed-lemma2} hold. Assume $g(\bm{x}_k)<0$ and for the time step size $\tau\leq \min\{\tau_-^*,\tau_+^*\}$, the proposed discrete solution $\tilde{\bm{x}}_{k+1}:=\hat{\bm{x}}(\tau)$ is well-defined and satisfies $g(\tilde{\bm{x}}_{k+1})>0$. Then there exists $\hat{t}^*\in (t_k,t_{k+1})$ and $\hat{\bm{x}}^*:=\hat{\bm{x}}(\hat{t}^*-t_k)$ such that $g(\hat{\bm{x}}^*)=0$; that is $(\hat{t}^*,\hat{\bm{x}}^*)$ satisfy equations \eqref{eq:to-interface} and \eqref{eq:interface}. Moreover, there exists $\bm{x}_{k+1}\in B_\delta(\hat{\bm{x}}^*)$ which satisfies equation \eqref{eq:next}.
\label{prop:existence_numerical_transition}
\end{proposition}
\begin{proof}
By Lemma \ref{lem:wellposed-lemma1}, $\hat{\bm{x}}:[0,\tau_-^*]\rightarrow B_\delta(\bm{x}_k)$ is continuous and so the function $h(t):=g(\hat{\bm{x}}(t-t_k))$ satisfies $h(t_k)=g(\hat{\bm{x}}(0))=g(\bm{x}_k)<0$ and $h(t_{k+1})=g(\hat{\bm{x}}(\tau))=g(\tilde{\bm{x}}_{k+1})>0$. So by Lemma \ref{lem:suff_cond_transition}, there exists a time $\hat{t}^*\in (t_k,t_{k+1})$ with $\hat{\bm{x}}^*:=\hat{\bm{x}}(\hat{t}^*-t_k)$ so that $g(\hat{\bm{x}}^*)=h(\hat{t}^*)=0$. The existence of $\bm{x}_{k+1}:=\bm{\hat{x}}(t_{k+1}-\hat{t}^*)\in B_\delta(\hat{\bm{x}}^*)$ follows from Lemma \ref{lem:wellposed-lemma2} and that $t_{k+1}-\hat{t}^*\leq \tau \leq \tau_+^*$.
\end{proof}

\section{Analysis of \texorpdfstring{$g(\bm{x}(t; \bm{y}, s))$}{piecewise smooth flow} and
  \texorpdfstring{$g(\bm{x}^\tau(t; \bm{y}, s))$}{discrete flow} near \texorpdfstring{$\bm{x}^*, t^*$}{x-star t-star}}\label{sec:signAnalysis}

In the proof of the main theorem, we needed to analyze the discrete and piecewise smooth trajectories from different initial data. The following two pairs of lemmas provide uniform estimates of Lemma~\ref{lem:part_conv_transition} for different initial data, one for piecewise smooth trajectories and one for discrete trajectories. Moreover, the following lemmas give geometric estimates for lengths of time intervals during which the smooth or discrete trajectories stay on one side of the interface. These intervals are inversely proportional in length to perturbed versions of the constant $M$ or $\hat{M}$ of~\eqref{eq:M} and~\eqref{eq:hatM} which depend on the Hessian (hence curvature) of the switching function $g$. The lengths of the intervals are directly proportional to $\alpha_S$ and $\hat{\alpha}_S$. These in turn are small when the vector fields are close to tangential to the switching surface, and larger when the vector fields are transversal to it.

\begin{lemma}\label{lem:ctsMonotonicTime}
    Let the transversality condition \eqref{eq:transverseA} hold at $(\bm{x}^\ast, t^\ast)$. Then there exists $\epsilon>0$ such that for all
    \begin{equation*}
        \bm{y} \in K^\epsilon_+ := \bar{B}_\epsilon(\bm{x}^\ast) \cap S,  \qquad (t,s) \in J^\epsilon_+ := \left\{ (t,s) \in [t^\ast, t^\ast + \epsilon]^2 ~|~ t > s \right\},
    \end{equation*}
    there are no transitions contained in the time interval $(s, t^\ast + \epsilon]$ for any trajectory $\bm{x}(t; \bm{y}, s)$.
    Furthermore, let $M^\epsilon_+$ be defined as
    \begin{equation}
        M^\epsilon_+ := \frac{1}{2} \sup_{ \substack{ (t,s) \in J^\epsilon_+ \\ \bm{y} \in K^\epsilon_+} } \left|\dot{\bm{x}}(t)\cdot  H_g(\bm{x}(t))\dot{\bm{x}}(t)+\nabla g(\bm{x}(t))\cdot \ddot{\bm{x}}(t)\right|. \label{eq:PerturbedM}
    \end{equation}
    If $0 < t - s < \min \left( t^\ast + \epsilon - s, \frac{\alpha_S^2}{M^\epsilon_+} \right)$, then
    \begin{equation}
        g(\bm{x}(t; \bm{y}, s)) > 0\, .
        \label{eq:ctsMonoLemma}
    \end{equation}
\end{lemma}
\begin{proof}
    The proof of \eqref{eq:PerturbedM} proceeds in three main steps: we first define an time interval where piecewise smooth trajectories starting on the interface are guaranteed to stay in $U_+$. After that, we show $M_+^\epsilon$ is well-defined. Finally, we introduce a key estimate of the lemma and use it to prove an explicit bound on time. 
    
    The first step of this proof has similarities to Lemma~\ref{lem:part_conv_transition}. However, whereas Lemma~\ref{lem:part_conv_transition} is a pointwise estimate for a specific initial data, here we derive an uniform estimate for a neighborhood of different initial data.  Define 
    $$F(t, s, \bm{y}) := \nabla g(\bm{x}(t; \bm{y}, s)) \cdot \bm{f}_+ (t, \bm{x}(t; \bm{y}, s)).$$ Note that $F$ is continuous in all its three arguments. This follows from the hypotheses that $g \in C^2$, $\bm{f}_+$ is continuous in all its arguments and $\bm{x}$ is assumed to be Lipschitz continuous in the initial data. Furthermore, by the transverality condition \eqref{eq:transverseA}, we have that $F(t^\ast, t^\ast, \bm{x}^\ast) \geq \alpha_S^2 > 0$. Then by continuity, there exists $\epsilon>0$ such that for $\bm{y} \in K^\epsilon_+$ and $(t,s) \in J^\epsilon_+$, $F(t, s, \bm{y}) \geq \frac{\alpha_S^2}{2}$. This guarantees that for all initial data $(\bm{y},s)\in K^\epsilon_+ \times [t^\ast, t^\ast + \epsilon]$, the trajectory $\bm{x}(t; \bm{y}, s)$ remains in $U_+$ for $t\in (s,t^*+\epsilon]$. Moreover, we now have a uniform $\epsilon$ to be used in the definition of $M^\epsilon_+$. 
    
    Next, we verify the well-definedness of $M_+^\epsilon$. First, by Hypothesis~\ref{hyp:H2}, the one-sided limit $\lim_{t \downarrow s} \dot{\bm{x}}(t; \bm{y}, s)$ exists. Second, since $\bm{x}(t; \bm{y}, s)$ is continuous, the supremum of $H_g$ can be taken over the closed interval $[s, t^\ast + \epsilon]$, which is bounded. Finally, the second derivative term $\ddot{\bm{x}}(t; \bm{y}, s)$ is bounded by Hypothesis~\ref{hyp:H3} and that there are no transitions on the interval $(s, t^\ast + \epsilon]$. Thus, $M_+^\epsilon$ is well-defined.
    
    For the last step to prove \eqref{eq:ctsMonoLemma}, since $g(\bm{x}(s; \bm{y}, s))=g(\bm{y})=0$, we have
    \begin{equation}
    \begin{split}
        g(\bm{x}(t; \bm{y}, s)) &= \nabla g(\bm{y})\cdot \bm{f}_+(s,\bm{y})(t-s)\\
        &\quad+\left(\frac{}{}g(\bm{x}(t; \bm{y}, s))-\left[g(\bm{x}(s; \bm{y}, s)) +\nabla g(\bm{x}(s; \bm{y}, s)) \cdot \bm{f}_+(s, \bm{y}) \right] (t-s)\right).
    \end{split} \label{eqn:signExp}
    \end{equation}
    Since $(t,s) \in J_+^\epsilon$, this implies $t-s < t^* + \epsilon - s$. Moreover if $0 < t - s < \frac{\alpha_S^2}{M_+^\epsilon}$, then by \eqref{lem:ctsTimeErrorTaylorTime} with $M$ replaced with $M_+^\epsilon$, \eqref{eqn:signExp} becomes
    \[
        g(\bm{x}(t; \bm{y}, s)) &\geq \underbrace{\nabla g(\bm{y}) \cdot \bm{f}_+(s, \bm{y})}_{\geq \alpha_S^2}(t - s) - M_+^\epsilon(t - s)^2\\ &
        \geq \underbrace{(t - s)}_{>0}\underbrace{(\alpha_S^2 - M_+^\epsilon(t - s))}_{>0} > 0.
    \]
    Taking $t-s$ less than the minimum of $\frac{\alpha_S^2}{M_+^\epsilon}$ and $t^* + \epsilon - s$ gives the desired result.
\end{proof}
Next, we show another version of Lemma~\ref{lem:ctsMonotonicTime} for $s > t$. Since the analysis is similar, we only highlight the differences in the proof.

\begin{lemma}\label{lem:ctsMonotonicTime2}
    Let the transversality condition \eqref{eq:transverseA} hold at $(\bm{x}^\ast, t^\ast)$. Then there exists $\epsilon>0$ such that for all
    \begin{equation*}
        \bm{y} \in K^\epsilon_- := \bar{B}_\epsilon(\bm{x}^\ast) \cap S,  \qquad (t,s) \in J^\epsilon_- := \left\{ (t,s) \in [t^\ast - \epsilon, t^\ast]^2 ~|~ t < s \right\},
    \end{equation*}
    there are no transitions contained in the time interval $[t^\ast - \epsilon, s)$ for any trajectory $\bm{x}(t; \bm{y}, s)$.
    Furthermore, let $M^\epsilon_-$ be defined as
    \begin{equation}
        M^\epsilon_- := \frac{1}{2} \sup_{ \substack{ (t,s) \in J^\epsilon_- \\ \bm{y} \in K^\epsilon_-} } \left|\dot{\bm{x}}(t)\cdot  H_g(\bm{x}(t))\dot{\bm{x}}(t)+\nabla g(\bm{x}(t))\cdot \ddot{\bm{x}}(t)\right|. \label{eq:PerturbedM2}
    \end{equation}
    If $\max \left( t^* - \epsilon - s, \frac{ - \alpha_S^2}{M^\epsilon_-} \right) < t - s < 0$, then
    \begin{equation}
        g(\bm{x}(t; \bm{y}, s)) < 0\, .
        \label{eq:ctsMonoLemma2}
    \end{equation}
\end{lemma}
\begin{proof}
    The first two steps of this proof follow the same format as the proof of Lemma~\ref{lem:ctsMonotonicTime}, with care being taken that we are dealing with times $t,s < t^*$. Moreover, the domain becomes $U_-$ instead of $U_+$. As shown in the proof of the previous Lemma~\ref{lem:ctsMonotonicTime}, the constants $\epsilon, K_-^\epsilon, J_-^\epsilon$ and $M_-^\epsilon$ are well-defined. Finally, to prove \eqref{eq:ctsMonoLemma2}, we utilize a similar estimate to \eqref{eqn:signExp}, only for $t<s$ and $\bm{f}_+$ replaced with $\bm{f}_-$. Since $(t,s) \in J_-^\epsilon$, this implies $t^* - \epsilon - s < t-s $. Moreover if $\frac{- \alpha_S^2}{M_-^\epsilon} < t-s < 0$, then by \eqref{lem:ctsTimeErrorTaylorTime}, with $M$ replaced with $M_-^\epsilon$, \eqref{eqn:signExp} becomes
    \[
        g(\bm{x}(t; \bm{y}, s)) &\leq \underbrace{- \nabla g(\bm{y}) \cdot \bm{f}_-(s, \bm{y})}_{\leq - \alpha_S^2} (s - t) + M_-^\epsilon(s - t)^2\\ &
        \leq \underbrace{(s - t)}_{>0}\underbrace{( M_-^\epsilon(s - t) - \alpha_S^2 )}_{<0} < 0.
    \] Similarly, taking $t-s$ greater than the maximum of $\frac{-\alpha_S^2}{M_-^\epsilon}$ and $t^* - \epsilon - s$ yields the desired result.
\end{proof}

We now prove analogous statements for discrete trajectories, in a similar manner as in the proof of Lemmas~\ref{lem:ctsMonotonicTime} and \ref{lem:ctsMonotonicTime2}.

\begin{lemma}\label{lem:discMonotonicTime}
    Let $\tau_1$, $C_1$, $C_2$ and $\hat{\alpha}_S^2$ be positive constants defined in Lemma~\ref{lem:discTC}. Then there exists $\epsilon_1, \epsilon_2>0$ such that for all
    \begin{equation*}
        \bm{y} \in \hat{K}_+^{\epsilon_1} := \bar{B}_{\epsilon_1}(\bm{x}^\ast) \cap S , \qquad (t,s) \in \hat{J}_+^{\epsilon_2} := \left\{ (t,s) \in [t^\ast, t^\ast + \epsilon_2]^2 ~ | ~ t > s \right\},
    \end{equation*}
    there are no transitions contained in the time interval $(s, t^\ast + \epsilon_2]$ for any trajectory $\bm{x}^\tau(t; \bm{y}, s)$. Furthermore, let $\hat{M}_+^{\epsilon_1, \epsilon_2}$ be defined as
    \begin{equation*}
        \hat{M}^{\epsilon_1, \epsilon_2}_+ := \sup_{\substack{(t,s) \in \hat{J}_+^{\epsilon_2} \\ \eta \in [0,1] \\ \bm{y} \in \hat{K}_+^{\epsilon_1}}}
        \lvert \bm{f}^\tau_+(s, \bm{y}, t, \bm{x}^\tau(t)) \cdot H_g( \bm{y} + \eta (\bm{x}^\tau(t) - \bm{y}) ) \bm{f}^\tau_+(s, \bm{y}, t, \bm{x}^\tau(t)) \rvert \, .\label{eq:PerturbedHatM}
    \end{equation*}
    If $0 < t - s < \min \left( \hat{t}^\ast + \epsilon - s, \frac{\hat{\alpha}_S^2}{\hat{M}^{\epsilon_1, \epsilon_2}_+} \right)$, then
    \begin{equation}
        g(\bm{x}^\tau(t; \bm{y}, s)) > 0
        \label{eq:disMonoLemma}
    \end{equation}
\end{lemma}
\begin{proof}
    The proof proceeds similarly in three steps; namely, we first define a time interval where discrete trajectories starting on the interface are guaranteed to stay in $U_+$. Next, we show that $\hat{M}^{\epsilon_1,\epsilon_2}_+$ is well-defined. Finally, we introduce a key estimate of the lemma and use it to prove an explicit bound on time. Note that the key estimate in this discrete case is subtly different than in the previous lemma.

    For the first step, let $\tau_1$, $C_1$, $C_2$ and $\alpha_S^2$ be the positive constants defined in the discrete transversality condition~\ref{lem:discTC}. Note that the function $F$ used in Lemma~\ref{lem:ctsMonotonicTime} is directly analogous to the discrete transversality condition. Let $\epsilon_1 := \min(C_1 \tau_1, C_2 \tau_1^p)$. By continuity of the discrete trajectory in time, there exists a positive $\epsilon_2 \leq \epsilon_1$ such that for $\bm{y} \in B_{\epsilon_1}(\bm{x}^\ast) \cap S$ and $s \in [t^\ast, t^\ast + \epsilon_2]$, then $\bm{x}^\tau(t; \bm{y}, s) \in B_{\epsilon_1}(\bm{x}^\ast)$ if $t \in (s, t^\ast + \epsilon_2]$. Then the discrete transversality condition holds  and $\bm{x}^\tau(t; \bm{y}, s) \in B_{\epsilon_1}(\bm{x}^\ast) \cap U_+$.
    
    Next, we begin the second step of the proof by verifying the well-definedness of $\hat{M}^{\epsilon_1, \epsilon_2}_+$. First, the supremum of $H_g$ can be taken over the compact set $(\bm{y}, \eta, t) \in \hat{K}_{\epsilon_1} \times [0,1] \times [s, t^\ast + \epsilon_2]$ and hence is well-defined and bounded. Next, since by the first step, $\bm{x}^\tau(t; \bm{y}, s) \in U_+$ for all $\bm{y} \in \hat{K}_{\epsilon_1}$ and $(t,s) \in \hat{J}_{\epsilon_2}$. Then by consistency of $\bm{f}_+^\tau$, we can bound $\norm{\bm{f}_+^\tau(s, \bm{y}, t, \bm{x}^\tau(t))}$ for all $\bm{y} \in \hat{K}_{\epsilon_1}$ and $(t,s) \in \hat{J}_{\epsilon_2}$ by 
    \[
        \norm{\bm{f}_+^\tau(s, \bm{y}, t, \bm{x}^\tau(t))} &\leq \norm{\bm{f}_+(t^\ast, \bm{x}^\ast) + \bm{f}_+^\tau(s, \bm{y}, t, \bm{x}^\tau(t))} + \norm{\bm{f}_+(t^\ast, \bm{x}^\ast)} \\
        &\leq C_3 \tau_1^p + \norm{\bm{f}_+(t^\ast, \bm{x}^\ast)} < \infty
    \]
    Thus, combining with the Cauchy Schwartz inequality and properties of induced operator norm in $\ell_2$, the supremum taken in $\hat{M}^{\epsilon_1, \epsilon_2}_+$ is bounded above by $\norm{\bm{f}_+^\tau}^2 \norm{H_g} < \infty$.
    
    For the last step to prove \ref{eq:disMonoLemma}, note the following which holds for $t>s$:
    \begin{equation}
        \begin{split}
            g(\bm{x}^\tau&(t; \bm{y}, s)) \\
            &= \nabla g(\bm{y})\cdot \bm{f}^\tau_+(s, \bm{y}, t, \bm{x}^\tau(t))(t-s)\\
            &\quad+\left(\frac{}{}g(\bm{x}^\tau(t))-\left[g(\bm{x}^\tau(s)) +\nabla g(\bm{x}^\tau(s))\cdot \bm{f}^\tau_+(s, \bm{y}, t, \bm{x}^\tau(t))\right](t-s)\right).
        \end{split}\label{eq:discSignExp}
    \end{equation}
    Since $(t,s) \in J_+^{\epsilon_2}$, this implies $t - s < t^* + \epsilon - s$. Moreover if $0 < t - s < \frac{\hat{\alpha}_S^2}{\hat{M}^{\epsilon_1, \epsilon_2}_+}$, then by~\eqref{eq:discGX-} with $\hat{M}$ replaced by $\hat{M}^{\epsilon_1, \epsilon_2}_+$, ~\eqref{eq:discSignExp} becomes
    
    \begin{equation*}
        \begin{split}
            g(\bm{x}^\tau(t; \bm{y}, s)) &\geq \underbrace{\nabla g(\bm{y}) \cdot \bm{f}^\tau_+(s, \bm{y}, t, \bm{x}^\tau(t))}_{\geq \hat{\alpha}^2_S} (t-s) - \hat{M}^{\epsilon_1, \epsilon_2}_+(t-s)^2 \\
                &\geq \underbrace{(t-s)}_{> 0} \underbrace{\left( \hat{\alpha}^2_S - \hat{M}^{\epsilon_1, \epsilon_2}_+ (t - s) \right)}_{> 0}  > 0
        \end{split}
    \end{equation*}
    Taking $t-s$ less than the minimum of $t^* + \epsilon - s$ and $\frac{\hat{\alpha}_S^2}{\hat{M}^{\epsilon_1, \epsilon_2}_+}$ yields the desired result.
\end{proof}
Next, we show another version of Lemma~\ref{lem:discMonotonicTime} for $s>t$. Since the analysis is once again similar, we only highlight the differences in the proof.

\begin{lemma}\label{lem:discMonotonicTime2}
    Let $\tau_1$, $C_1$, $C_2$ and $\hat{\alpha}_S^2$ be positive constants defined in Lemma~\ref{lem:discTC}. Then there exists $\epsilon_1, \epsilon_2>0$ such that for all
    \begin{equation*}
        \bm{y} \in \hat{K}_-^{\epsilon_1} := \bar{B}_{\epsilon_1}(\bm{x}^\ast) \cap S , \qquad (t,s) \in \hat{J}_-^{\epsilon_2} := \left\{ (t,s) \in [t^\ast - \epsilon_2, t^\ast]^2 ~ | ~ t < s \right\},
    \end{equation*}
    there are no transitions contained in the time interval $[t^\ast - \epsilon_2, s)$ for any trajectory $\bm{x}^\tau(t; \bm{y}, s)$. Furthermore, let $\hat{M}_-^{\epsilon_1, \epsilon_2}$ be defined as
    \begin{equation*}
        \hat{M}^{\epsilon_1, \epsilon_2}_- := \sup_{\substack{(t,s) \in \hat{J}_-^{\epsilon_2} \\ \eta \in [0,1] \\ \bm{y} \in \hat{K}_-^{\epsilon_1}}}
        \lvert \bm{f}^\tau_-(s, \bm{y}, t, \bm{x}^\tau(t)) \cdot H_g( \bm{y} + \eta (\bm{x}^\tau(t) - \bm{y}) ) \bm{f}^\tau_-(s, \bm{y}, t, \bm{x}^\tau(t)) \rvert \, .\label{eq:PerturbedHatM2}
    \end{equation*}
    If $\max \left( \hat{t}^\ast - \epsilon - s, \frac{- \hat{\alpha}_S^2}{\hat{M}^{\epsilon_1, \epsilon_2}_-} \right) < t - s < 0$, then
    \begin{equation}
        g(\bm{x}^\tau(t; \bm{y}, s)) < 0 \label{eq:disMonoLemma2}
    \end{equation}
\end{lemma}
\begin{proof}
    First, by Lemmas~\ref{lem:wellposed-lemma1} and \ref{lem:wellposed-lemma2}, the discrete forward flow is well-defined. Thus, by replacing $\bm{f}_+^\tau$ with $\bm{f}_-^\tau$ and integrating backward in time, the backward flow is well-defined. So $\bm{x}^\tau(t; \bm{y}, s)$ is well-defined for $t<s$.
    
    The next two steps of this proof follow the same format as the proof of Lemma~\ref{lem:discMonotonicTime}, with care being taken that we are dealing with time $t,s < t^*$. Moreover, the domain becomes $U_-$ instead of $U_+$. As shown in the proof of the previous Lemma~\ref{lem:discMonotonicTime}, the constants $\epsilon_1, \epsilon_2, K_-^{\epsilon_1}, J_-^{\epsilon_2}$ and $M_-^{\epsilon_1, \epsilon_2}$ are well-defined. Finally, to prove \eqref{eq:disMonoLemma2}, we utilize an estimate similar to \eqref{eq:discSignExp}, only for $t < s$ and $\bm{f}_+^\tau$ replaced with $\bm{f}_-^\tau$. Since $(t,s) \in J_-^{\epsilon_2}$, this implies $t^* - \epsilon - s < t-s$. Moreover if $\frac{- \alpha_S^2}{M_-^\epsilon} < t-s < 0$, then by \eqref{lem:ctsTimeErrorTaylorTime}, with $M$ replaced with $M_-^\epsilon$, \eqref{eq:discSignExp} becomes
    
    \begin{equation*}
        \begin{split}
            g(\bm{x}^\tau(t; \bm{y}, s)) &\leq \underbrace{- \nabla g(\bm{y}) \cdot \bm{f}^\tau_-(t, \bm{x}^\tau(t), s, \bm{y})}_{\leq - \hat{\alpha}^2_S} (s-t) - \hat{M}_-^{\epsilon_1, \epsilon_2}(s-t)^2 \\
                &\leq \underbrace{(s - t)}_{> 0} \underbrace{\left( \hat{M}_-^{\epsilon_1, \epsilon_2} (s - t) - \hat{\alpha}^2_S \right)}_{< 0} < 0
        \end{split}
    \end{equation*}
    Taking $t-s$ greater than the maximum of $t^* - \epsilon - s$ and $\frac{- \hat{\alpha}_S^2}{\hat{M}^{\epsilon_1, \epsilon_2}_-}$ yields the desired result.
\end{proof}

\end{document}